\documentclass[]{amsart}
\usepackage[utf8]{inputenc}
\usepackage[mathscr]{eucal} 
\usepackage{stmaryrd} 
\usepackage{amsmath,amsthm,amsfonts,amssymb,amscd} 
\usepackage[dvipsnames]{xcolor} 
\usepackage{xspace} 
\usepackage{fancyhdr} 
\usepackage{graphicx} 
\usepackage{listings} 
\usepackage[]{hyperref} 
\usepackage{enumitem} 
\usepackage{relsize} 
\usepackage{tikz-cd} 
\usepackage{mdwlist} 
\usepackage{multicol} 
\usepackage{float} 
\usepackage{adjustbox} 
\usepackage{tikz} 
\usepackage{changepage} 
\usepackage[noabbrev,nameinlink]{cleveref} 
\Crefformat{section}{#2\S#1#3}
\usepackage{setspace} 
\usepackage{bbm} 
\usetikzlibrary{matrix, calc, arrows}

\hypersetup{
    colorlinks=true, 
    linktoc=all,     
    linkcolor=Brown,citecolor=Brown,urlcolor=MidnightBlue,  
}

\setlist[enumerate]{label=(\alph*)} 
\usetikzlibrary{graphs,decorations.pathmorphing,decorations.markings}
\tikzcdset{scale cd/.style={every label/.append style={scale=#1},
        cells={nodes={scale=#1}}}}

\newcommand{\Hom}{\operatorname{Hom}}

\newcommand{\inv}{^{-1}}
\newcommand{\Ob}{\operatorname{Ob}}

\newcommand{\id}{\operatorname{id}}
\newcommand{\Id}{\operatorname{Id}}

\newcommand{\stab}{\operatorname{stab}}

\newcommand{\Spc}{\operatorname{Spc}}

\newcommand{\pt}{\operatorname{pt}}
\newcommand\isoto{\stackrel{\sim}{\smash{\longrightarrow}\rule{-1pt}{0.4ex}}}

\newcommand{\on}[1]{\operatorname{#1}}

\newcommand{\calD}{\mathcal{D}}

\newcommand{\catF}{\mathscr{F}}
\newcommand{\catG}{\mathscr{G}}

\newcommand{\catI}{\mathscr{I}}
\newcommand{\catJ}{\mathscr{J}}
\newcommand{\catK}{\mathscr{K}}
\newcommand{\catL}{\mathscr{L}}
\newcommand{\catM}{\mathscr{M}}
\newcommand{\catN}{\mathscr{N}}

\newcommand{\catP}{\mathscr{P}}

\newcommand{\catT}{\mathscr{T}}

\newcommand{\fraki}{\mathfrak{i}}

\newcommand{\frakp}{\mathfrak{p}}

\newcommand{\bbF}{\mathbb{F}}
\newcommand{\bbG}{\mathbb{G}}

\newcommand{\bbM}{\mathbb{M}}
\newcommand{\bbN}{\mathbb{N}}

\newcommand{\bbR}{\mathbb{R}}

\newcommand{\bbone}{\mathbbm{1}}

\newcommand{\supp}{\on{supp}}
\newcommand{\sbull}{{\scriptscriptstyle\bullet}}

\newcommand{\Ext}{\on{Ext}}
\newcommand{\Spec}{\on{Spec}}
\newcommand{\cp}{^{cp}}

\newcommand{\Proj}{\on{Proj}}
\newcommand{\down}{\downarrow\hspace{-0.28em}}

\newtheorem{theorem}{Theorem}[section]
\newtheorem{lemma}[theorem]{Lemma}
\newtheorem{proposition}[theorem]{Proposition}
\newtheorem{corollary}[theorem]{Corollary}

\newtheorem*{theorem*}{Theorem}

\newtheorem*{question}{Question}

\newtheorem{innercustomthm}{Theorem}
\newenvironment{customthm}[1]
{\renewcommand\theinnercustomthm{#1}\innercustomthm}
{\endinnercustomthm}

\theoremstyle{remark}
\newtheorem{remark}[theorem]{Remark}
\newtheorem{example}[theorem]{Example}

\newtheorem{observation}[theorem]{Observation}

\theoremstyle{definition}
\newtheorem{definition}[theorem]{Definition}
\newtheorem{notation}[theorem]{Notation}

\newtheorem*{definition*}{Definition}

\AddToHook{env/corollary/begin}{\crefalias{theorem}{corollary}}
\AddToHook{env/proposition/begin}{\crefalias{theorem}{proposition}}
\AddToHook{env/lemma/begin}{\crefalias{theorem}{lemma}}
\AddToHook{env/definition/begin}{\crefalias{theorem}{definition}}
\AddToHook{env/example/begin}{\crefalias{theorem}{example}}
\AddToHook{env/remark/begin}{\crefalias{theorem}{remark}}
\AddToHook{env/observation/begin}{\crefalias{theorem}{observation}}

\crefname{theorem}{Thm.}{Thms.}
\crefname{proposition}{Prop.}{Props.}
\crefname{lemma}{Lem.}{Lem.}
\crefname{corollary}{Cor.}{Cors.}
\Crefname{theorem}{Theorem}{Theorems}
\Crefname{proposition}{Proposition}{Propositions}
\Crefname{lemma}{Lemma}{Lemmas}
\Crefname{corollary}{Corollary}{Corollaries}

\makeatletter
\newcommand{\xhookdoubleheadrightarrow}[2][]{%
    \lhook\joinrel
    \ext@arrow 0359\rightarrowfill@ {#1}{#2}%
    \mathrel{\mspace{-15mu}}\rightarrow
}
\makeatother

\begin{document}
    \title[Re-framing nctt-geometry]{Re-framing the classification of ideals in noncommutative tensor triangular geometry}
    \date{\today}

    \author[T.~De Deyn]{Timothy De Deyn}
    \address{T.~De Deyn,
        School of Mathematics and Statistics,
        University of Glasgow,
        Glasgow G12 8QQ,
        United Kingdom}
    \curraddr{D\'{e}partement de Math\'{e}matique, Universit\'{e} Libre de Bruxelles, Boulevard du Triomphe, B-1050 Brussels, Belgium}
    \email{timothy.de.deyn@ulb.be}
    \urladdr{https://tdedeyn.github.io/}

    \author[S.~K.~Miller]{Sam K. Miller}
    \address{S.~K.~Miller,
        Department of Mathematics, University of Georgia, Athens GA 30602, United States of America} 
    \email{sam.miller@uga.edu} 
    \urladdr{https://www.samkmiller.com/}

    \begin{abstract}
        We prove that a `pseudo-Hochster dual' of the noncommutative Balmer spectrum of a monoidal triangulated category classifies its semiprime thick $\otimes$-ideals. This extends Balmer's classification of radical thick $\otimes$-ideals to noncommutative tensor triangular geometry.
        To achieve this, we utilize the notion of support data for lattices and frames, under which the classification follows via Stone duality.
        We also give a characterization for when the noncommutative Balmer spectrum behaves as it does in tensor triangular geometry, that is, when it is a spectral space with quasi-compact opens given by complements of supports.
        Finally, we show that rigid centrally generated monoidal triangulated categories satisfy this property, and we answer a question posed by Negron--Pevtsova regarding classification of one-sided $\otimes$-ideals via cohomological support.

    \end{abstract}

    \maketitle
    \section*{Introduction}

    Given a mountain, a mountaineer's burning desire is to summit it, and given a tensor triangulated category, a tensor triangular geometer's burning desire is to classify its thick $\otimes$-ideals. Such classification questions have presented themselves for a few decades now in widely varying fields, in differing languages and with differing motivations. Classical classification results in stable homotopy theory \cite{DHS88, HS98}, algebraic geometry \cite{Hop87,Nee96,Tho97}, and modular representation theory \cite{BCR97, FP07} were united under Balmer's theory of tensor triangular geometry \cite{Bal05}. In this setting, one associates to an essentially small tensor triangulated category $\catK$ a spectral topological space, now christened the \textit{Balmer spectrum} $\Spc(\catK)$, equipped with, for each object $x \in \catK$, its \textit{support} $\supp(x)$, a closed subset with quasi-compact complement. Together, this data determines the \textit{universal support theory} of $\catK$.

    One of Balmer's key insights \cite[Theorem 4.10]{Bal05} fundamentally connects $\Spc(\catK)$ to the classification of thick $\otimes$-ideals: one has an order-preserving bijection between the \textit{radical} (an automatic assumption if for instance $\catK$ is \textit{rigid}) thick $\otimes$-ideals of $\catK$ and the so-called \textit{Thomason subsets} of $\Spc(\catK)$, that is, unions of closed subsets with quasi-compact complement. Moreover, such a so-called \textit{classifying support} is unique - if a topological space with support similarly classifies the radical thick $\otimes$-ideals of $\catK$, it must in fact be $\Spc(\catK)$ \cite{Bal05, BKS07}.

    As the world is fundamentally noncommutative, humanity need not limit itself to a symmetric tensor product. Nakano--Vashaw--Yakimov \cite{NVY22}, building on Buan--Krause--Solberg's approach \cite{BKS07}, formalized the theory of \textit{noncommutative} tensor triangular geometry, the study of \textit{monoidal}-triangulated categories.
    Here, one has a \textit{noncommutative Balmer spectrum} $\Spc(\catK)$, as well as a \textit{complete prime spectrum} $\Spc\cp(\catK)$, again with associated supports on objects. Again, $\Spc(\catK)$ is the universal \textit{noncommutative} support datum, and Miller showed analogously that $\Spc\cp(\catK)$ is the universal \textit{multiplicative} support datum \cite{Mil25}. But life without a commutative tensor product can be vexing, and the question of classification of thick $\otimes$-ideals was left unresolved. The first answer came from Nakano--Vashaw--Yakimov in their pioneering paper. The authors determined that if a space classifies thick $\otimes$-ideals under certain hypotheses (including rigidity, extension to a ``big'' category, and Noetherianity of the classifying space), then indeed the space must be $\Spc(\catK)$. Other examples of when $\Spc(\catK)$ classifies \textit{semiprime} thick $\otimes$-ideals (the noncommutative analogue of radical thick $\otimes$-ideals, which again becomes automatic if $\catK$ is rigid) subsequently followed:
    \begin{enumerate}
        \item when $\catK$ has a thick generator \cite{NVY23};
        \item when $\Spc(\catK) = \Spc\cp(\catK)$ \cite{MR23};
        \item when $\Spc(\catK)$ is Noetherian \cite{Row24}.
    \end{enumerate}

    However, all hope for a classification theorem mirroring the commutative case was lost when Huang--Vashaw dropped a bombshell \cite[Example 7.7]{HV25}: an example of a monoidal triangulated category $\catK$ whose (semiprime) thick $\otimes$-ideals are parametrized by the \textit{specialization-closed} subsets of $\Spc(\catK)$, which in this example is a set strictly larger than the collection of Thomason subsets of $\Spc(\catK)$. Suddenly, the picture of how classification should look, or even if it was always possible, was rather murky.

    A different perspective on the classification was given in Buan--Krause--Solberg \cite{BKS07} who suggested a lattice-theoretic approach which was further developed by Koch--Pitsch \cite{KP17}.
    The key insight is that in tensor triangular settings, the Hochster dual $\Spc(\catK)^\vee$ of the Balmer spectrum is the spectral space corresponding to the \emph{coherent frame} of radical thick $\otimes$-ideals via Stone duality, an antiequivalence between certain topological spaces and frames \cite{Sto39}. Classification of radical thick $\otimes$-ideals is therefore a consequence of the duality. Gratz--Stevenson \cite{GS23} and Krause \cite{K23, Kra24} extended the lattice theory point of view further, demonstrating its potential use in wider contexts, such as tensor \textit{exact} geometry and noncommutative tensor triangular geometry (although \cite{BKS07} already had sown the noncommutative seeds). This perspective has been adopted in practical contexts recently, see e.g., \cite{Aok23, BKS20, BGLS26, BS21}.

    In this paper, we further harness Stone duality to settle the following looming classification questions of noncommutative tensor triangular geometry:

    \begin{question}\hfill
        \begin{enumerate}
            \item Can we always classify the semiprime thick $\otimes$-ideals of an essentially small monoidal triangulated category $\catK$, given its universal support $(\Spc(\catK), \supp)$?
            \item When does classification behave as in the tensor triangular setting? That is, $\Spc(\catK)$ is a spectral topological space with quasi-compact open sets given by the complements of supports, and we have an order-preserving bijection between the semiprime thick $\otimes$-ideals of $\catK$ and the Thomason subsets of $\Spc(\catK)$?
        \end{enumerate}
    \end{question}

    The answer to the first question is, in short, \textbf{yes}! We denote by $T_s(\catK)$ the lattice of semiprime thick $\otimes$-ideals of $\catK$, and given a topological space $X$, we denote the lattice of open subsets of $X$ by $\Omega(X)$. We view \cite[Theorem 6.4.5]{GS23} as a precursor and inspiration to this theorem, as it observes that the noncommutative Balmer spectrum is recovered from $T_s(\catK)$ under the hypothesis that all ideals are semiprime.

    \begin{customthm}{A}[\Cref{cor:classification_for_spatial_frames}]\label{thm:thmA}
        Let $\catK$ be an essentially small monoidal triangulated category. We have a natural isomorphism of lattices \[T_s(\catK) \cong \Omega(\Spc(\catK)^\nu),\] where $\Spc(\catK)^\nu$ denotes the ``pseudo-Hochster dual'' of $(\Spc(\catK),\supp)$. That is, $\Spc(\catK)^\nu = \Spc(\catK)$ as a set with supports $\supp(x)$, $x\in \catK$, forming an \textbf{open} base for the topology.
    \end{customthm}

    We note that this ``pseudo-Hochster dual'' is not necessarily the true Hochster dual, since $\Spc(\catK)$ and $\Spc(\catK)^\nu$ need not be spectral, but one can still obtain one topology from the other as long as the support data is known -- the supports describe a closed basis of $\Spc(\catK)$, and an open basis of $\Spc(\catK)^\nu$.

    In particular, as a consequence of \Cref{thm:thmA} we reconcile Huang--Vashaw's `counterexample' to classification: in their example the specialization-closed subsets of $\Spc(\catK)$ are exactly the open subsets of $\Spc(\catK)^\nu$. See \Cref{obs:Ts_can_be_ugly}.

    \Cref{thm:thmA} answers the first question, so far so good.
    For the second, Huang--Vashaw's example shows that the answer cannot be ``always.''
    Indeed, the frame of semiprime thick $\otimes$-ideals is spatial, but not always coherent.
    Recall that a principal semiprime thick $\otimes$-ideal is one generated by a single element; so of the form $\langle x\rangle_s$, where the notation stands for the smallest semiprime thick $\otimes$-ideal containing $x\in\catK$.
    In the commutative case these always form a compact generating sublattice of the lattice of semiprime (i.e.,\ radical) thick $\otimes$-ideals.
    In the noncommutative case, however, these may not form a sublattice of $T_s(\catK)$ or may not consist of compact elements.
    We denote the sub-poset of principal semiprime thick $\otimes$-ideals by $t_s(\catK)$, and call it the \textit{principal part} of $T_s(\catK)$.
    Re-framing the second question lattice-theoretically, the equivalent question becomes:

    \begin{question}
        When is $T_s(\catK)$ a coherent frame, with its sublattice of compact elements precisely $t_s(\catK)$, the sub-poset of the principal semiprime thick $\otimes$-ideals?
    \end{question}

    In order to answer the latter question we introduce the following definitions.

    \begin{definition*} Let $\catK$ be an essentially small monoidal triangulated category.
        \begin{enumerate}
            \item We say $\catK$ is \textit{principally closed} if for all $x, y \in\catK$, we have $\langle x\rangle_s \cap \langle y\rangle_s = \langle z\rangle_s$ for some $z \in \catK$.
            \item We say $\catK$ has \textit{compact detection} if for all $x \in \catK$, there exists a $s_x \in \catK$ such that the following holds: for all semiprime thick $\otimes$-ideals $\catI \in T_s(\catK)$, $x \in \catI$ if and only if $x \otimes s_x \otimes x \in \catI$.
        \end{enumerate}
    \end{definition*}

    Compact detection is a mild assumption; it holds for instance if $\catK$ is half-rigid. More generally, it holds when all ideals are semiprime, see \Cref{cor:semiprimeimpliescompactdetection}, and when all semiprime thick $\otimes$-ideals are radical. However, it is not automatic, see \Cref{ex:non-compact_det}. Principal closure is more restrictive, but still holds for all the examples where Balmer classification is known, and additionally, \textit{(weakly) centrally generated} categories; a fact shown in  \Cref{sec:centralgeneration}. The condition of central generation has come to the foregrounds of noncommutative tensor triangular geometry \cite{NP23, NVY25}, and Negron--Pevtsova showed the condition occurs for stable module categories of certain (non-co-commutative) Hopf algebras.

    We have that $\catK$ is principally closed if and only if $t_s(\catK)$ is a sublattice of $T_s(\catK)$, and is precisely the poset of compact elements if and only if $\catK$ has compact detection. Consequently, we obtain the following theorem, thereby answering the second question.

    \begin{customthm}{B}[\Cref{thm:distlattice}, \Cref{cor:ncspcspectralspace}, and \Cref{cor:balmerclassification}]\label{thm:thmB}
        Let $\catK$ be a monoidal triangulated category. The following are equivalent:
        \begin{enumerate}
            \item $\catK$ has principal closure and compact detection;
            \item The poset $t_s(\catK)$ is a bounded distributive lattice consisting of the compact elements of $T_s(\catK)$;
            \item The noncommutative Balmer spectrum $\Spc(\catK)$ is a spectral space with quasi-compact opens given by complements of supports,  \[K^\circ(\Spc(\catK)) = \{\supp(x)^c \mid x \in \catK\};\]

        \end{enumerate}
        If any of these conditions hold, we have an order-preserving bijection with the Thomason subsets
        \begin{align*}
            T_s(\catK) &\isoto \on{Th}(\Spc(\catK)), \\
            \catI\quad &\longmapsto  \bigcup_{x \in \catI} \supp(x).
        \end{align*}
    \end{customthm}

    In particular, all previously known examples where Balmer's classification holds noncommutatively satisfy principal closure and compact detection. We note that it is, in principle, still possible for $\Spc(\catK)$ to be spectral without the two conditions holding.
    However, for this to occur, compact detection must fail; in particular, $\catK$ cannot be rigid, so an example of this phenomenon would be rather exotic.

    With these classification results, we extend recent results of Barthel \cite{Bar25}, which themselves generalize a classical result of Cohen \cite{Coh50} to the lattice-theoretic and monoidal triangulated settings, see \Cref{sec:Cohen} and \Cref{cor:Barthel_translated}. We also strengthen \cite[Theorem 6.2.1]{NVY22} to detect non-Noetherian Balmer spectra. This is a noncommutative analogue of \cite[Theorem 3.1]{Del10}. See \Cref{def:weak_supp} for necessary terminology.

    \begin{customthm}{C}[\Cref{thm:classifyingsupportdatumforclosedsupport}]\label{thm:thmD}
        Let $\catK$ be an essentially small monoidal triangulated category which is principally closed and has compact detection, and let $(X, \sigma)$ be a weak support datum (in the sense of \cite{NVY22}) for $\catK$.
        \begin{enumerate}
            \item If $(X,\sigma)$ is injective and realizing, then $(X,\sigma)$ is a classifying support datum, hence $(X,\sigma) \cong (\Spc(\catK),\supp)$.
            \item Assume $\catK = \catT^c$ for a rigidly-compactly-generated monoidal triangulated category $\catT$ and that $(X,\sigma)$ is an extended weak support datum for $\catT$. If $(X,\sigma)$ is \textit{faithful} and realizing, then $(X,\sigma) \cong (\Spc(\catK), \supp)$.
        \end{enumerate}
        In either case, if $(X,\sigma)$ is Noetherian-realizing, then $\Spc(\catK)$ is Noetherian.

    \end{customthm}

    Finally, we work out examples in \Cref{sec:ex}, considering (weakly) centrally generated categories and well-behaved crossed product categories.
    We then put the abstract machinery to use to answer a question posed by Negron--Pevtsova \cite[Question 11.1]{NP23}. We obtain the classification of \textit{one-sided} thick $\otimes$-ideals for stable module categories of coordinate algebras of finite group schemes via cohomological support.
    We refer to \Cref{sec:ex} for precise definitions of the notions in the following statement.
    \begin{customthm}{D}[\Cref{cor:NP}]
        Let $\catG$ be a finite group scheme over a perfect field.
        There is a bijection, induced by a comparison map $\catM$ defined using cohomological support, between specialization-closed subsets of $\Proj(\Ext^\sbull_{\on{Coh}(\catG)}(\bbone, \bbone))$ and thick right $\otimes$-ideals of $\stab(\on{Coh}(\catG))$.
    \end{customthm}

    To close, let us mention that our main mantra (for which we claim no originality) is that support theory and Balmer's classification result in tensor triangular geometry can be re-framed in terms of Stone duality. Therefore, understanding classification in the noncommutative setting necessitates an understanding of the lattice theory underpinning $\otimes$-ideals.
    We give a table translating the various lattice- and support-theoretic properties.

    \begin{table}[H]
        \begin{adjustwidth}{-2cm}{-2cm}
            \centering
            \begin{tabular}{|c|c|c|}
                \hline
                Lattice properties & Support properties & Refs. \\
                \hline \hline
                $T_s(\catK)$ is coherent & $\Spc(\catK)^\vee$, $\Spc(\catK)$ are spectral &  \cref{prop:dist_lattice=coherent_frame} \\
                \hline
                $T_s(\catK)$ is spatial & $\Spc(\catK)^\nu$ is sober & \cref{thm:Stone}  \\
                \hline
                $\catI \in T_s(\catK)$ is compact & $\supp(\catI) \subseteq \Spc(\catK)^\nu$ is quasi-compact & \cref{prop:quasicompacts}\\
                \hline
                $t_s(\catK) = T_s(\catK)^c$ & $\catK$ has compact detection & \cref{cor:semiprimeimpliescompactdetection} \\
                \hline
                $t_s(\catK)$ is a sublattice of $T_s(\catK)$ & $\catK$ has principal closure & \cref{lem:distributivitygivenprincipalclosure}\\
                \hline
                All elements of $T_s(\catK)$ are compact &  \begin{tabular}{@{}c@{}} All ideals of $\catK$ are principal; \\ $\Spc(\catK)^\nu$ is Noetherian\end{tabular} & \cref{cor:Barthel_translated} \\
                \hline
            \end{tabular}
            \centering
            \caption{A translation of lattice-theoretic and support-theoretic properties.}
            \label{tab:placeholder}
        \end{adjustwidth}
    \end{table}

    \subsection*{Organization}
    {
        \crefformat{section}{#2{Section}~#1#3}
        \crefmultiformat{section}{#2{Sections}~#1#3}{ and~#2#1#3}{, #2#1#3}{, and~#2#1#3}
        \crefrangeformat{section}{#3{Sections} #1#4 through #5#2#6}

        This paper is divided into two parts. The first half of the paper, \cref{sec:lat_prelim,sec:support,sec:Cohen}, concern abstract lattice theory.
        In \cref{sec:lat_prelim}, we begin with some lattice-theoretic preliminaries, introducing frames and Stone duality. \cref{sec:support} concerns the notions of support data for lattices and frames as well as classifying support. Much of \cref{sec:support} regards direct analogues of familiar tensor triangular facts in the lattice context. \cref{sec:Cohen} briefly restates Barthel's adaptation of Cohen's theorem for lattices.

        The second half of the paper applies the first half to noncommutative tensor triangular geometry. \cref{sec:nctt} reviews some tensor triangular preliminaries and proves that the lattice of semiprime thick $\otimes$-ideals of a monoidal triangulated category is a spatial frame, hence proving \Cref{thm:thmA}. \cref{sec:coh_and_pp} considers principal closure, compact detection and deduces the main results including \Cref{thm:thmB} and \Cref{thm:thmD}. Finally, \cref{sec:ex} considers examples of principal closure and compact detection: we show that (weakly) centrally generated categories with compact detection satisfy principal closure, and that, under certain hypotheses, crossed product categories can inherit compact detection and principal closure. We end by discussing Negron--Pevtsova's question and comment on classification of one-sided thick $\otimes$-ideals.
    }

    \subsection*{Tool and computational resource disclosure}

    The authors affirm that generative AI (gen-AI) was not used in any way in either the research or writing stages of this paper, with one exception: \Cref{ex:non-compact_det} was discovered by Claude Fable during the review process. However, the example was written, along with rest of the paper, purely by the authors without assistance or input from gen-AI. We take full responsibility for the correctness of this example, as we do for the rest of the paper.

    \subsection*{Acknowledgments} The authors thank Henning Krause, Dan Nakano, Greg Stevenson, and Kent Vashaw for numerous helpful conversations and comments. We also thank the anonymous referee for a careful reading and valuable suggestions.
    De Deyn was supported by ERC Consolidator Grant 101001227 (MMiMMa). Miller was partially supported by an AMS-Simons travel grant.

    \section{Lattice preliminaries}\label{sec:lat_prelim}

    We start by recalling some lattice theoretic concepts. Our aim is to translate some standard tensor triangular notions to the context of bounded distributive lattices. This is not a new idea, see also \cite{KP17,GS23, Kra24}.
    For a comprehensive review of general lattice theory, we refer the reader to \cite{Joh82}; for spectral spaces, to \cite{DST19}. Note that \cite{Joh82} defines lattices to be bounded by default, we do not follow this convention.

    \subsection{Bounded distributive lattices and coherent frames}

    Recall a \emph{lattice} $\catL$ is a partially ordered set (poset) such that every non-empty finite subset of elements has least upper bound and greatest lower bound, i.e., it has non-empty finite joins and meets, denoted  $\vee$ and $\wedge$ respectively.
    The empty join and meet, if they exist, are called the \emph{bottom} and \emph{top} and are denoted $0$ and $1$; a lattice with a bottom and top is called \emph{bounded}.
    If joins distribute over meets, equivalently meets distribute over joins, the lattice is called \emph{distributive}. A morphism of (distributive) lattices is a poset homomorphism commuting with finite joins and meets.

    When $\catL$ has arbitrary joins it is called \emph{complete} (this is equivalent to having arbitrary meets).
    Moreover, if additionally arbitrary joins distribute over finite meets the lattice is called a \emph{frame}.
    A morphism of frames is a lattice morphism that commutes with arbitrary joins and finite meets.
    An element $x\in \catL$ is called \emph{compact} (also called \emph{finite}, \emph{finitely presented} or \emph{small}) if for all subsets $S \subseteq \catL$ with $x \leq \vee S$, there exists a finite subset $S' \subseteq S$ such that $x \leq \vee S'$.
    \begin{definition}
        A frame $\catF$ is called \emph{coherent} if the compact elements form a bounded (necessarily distributive) sublattice and generate $\catF$ under joins, i.e., any element of $\catF$ is join of compact elements.
        A morphism of coherent frames is a frame morphism that preserves compact elements.
    \end{definition}
    \begin{remark}
        When the compact elements generate under join, but they do not necessarily form a sublattice, the frame is called \emph{algebraic}.
    \end{remark}

    To any lattice one can associate another lattice, that of its ideals, which is a coherent frame when the starting lattice is bounded and distributive.
    An \textit{ideal} $\catI \subseteq \catL$ is a non-empty downward set that is closed under finite joins. That is, if $x \in \catI$ and $y \leq x$, then $y \in \catI$, and if $x', y'\in \catI$, then $x' \vee y' \in \catI$. The set of ideals is denoted $\Id(\catL)$; it is a bounded, complete lattice.
    An ideal $\catI$ of $\Spc(\catL)$ is \textit{principal} if it is generated by a single element $x \in \catL$. That is, $\catI=\{ y \in \catL \mid y \leq x\}$. We denote the ideal generated by $x$ by $\down x$. The principal ideals are precisely the compact elements of the ideal lattice $\Id(\catL)$ of $\catL$ and they generate under joins, demonstrating that $\Id(\catL)$ is coherent.
    In fact, any coherent frame is necessarily of this form.

    \begin{proposition}[{see e.g.,\ \cite[Corollary II.3.3]{Joh82}}] \label{prop:dist_lattice=coherent_frame}
        There is an equivalence of categories
        \[
        \begin{tikzcd}
            {\mathsf{BDLat}} & {\mathsf{CohFrm}}\rlap{,}
            \arrow["{\Id}", shift left=1, from=1-1, to=1-2]
            \arrow["(-)^c", shift left=1, from=1-2, to=1-1]
        \end{tikzcd}
        \]
        between the category $\mathsf{BDLat}$ of bounded distributive lattices and the category $\mathsf{CohFrm}$ of coherent frames.
        The functor $\catL\mapsto\Id(\catL)$ maps a bounded distributive lattice to its lattice of ideals, whilst $\catF\mapsto\catF^c$ maps a coherent frame to its sublattice of compact elements.
    \end{proposition}

    \subsection{Stone duality}
    Stone duality gives an antiequivalence between spatial lattices and sober topological spaces.
    Let us recall these.
    A topological space is \emph{sober} if every nonempty irreducible closed subset has a unique generic point; we denote by $\mathsf{Sob}$ the category of sober spaces and continuous maps.
    A point of a frame $\catF$ is a frame morphism $\catF\to \mathbf{2}$, where $\mathbf{2}$ denotes the two-element bounded lattice $\{0,1\}$.
    \begin{definition}
        A frame $\catF$ is called \emph{spatial} when it has enough points in the sense that whenever $x\not\leq y$ in $\catF$ there exists a point $p\colon \catF\to \mathbf{2}$ with $p(x) = 1$ and $p(y) = 0$.
        We denote by $\mathsf{SFrm}$ the full subcategory of $\mathsf{Frm}$ consisting of spatial frames.
    \end{definition}

    In \Cref{prop:frameptsaremeetprimes,prop:equivalent_def_of_spatial}, we will observe that points and spatiality translate into primality conditions.

    \begin{theorem}[{\cite[Stone duality]{Sto39}}]\label{thm:Stone}
        There is a duality
        \[
        \begin{tikzcd}
            {\mathsf{SFrm}^{\on{op}}} & {\mathsf{Sob}}\rlap{,}
            \arrow["{\pt}", shift left=1, from=1-1, to=1-2]
            \arrow["\Omega", shift left=1, from=1-2, to=1-1]
        \end{tikzcd}
        \]
        between the category $\mathsf{SFrm}$ of spatial frames and the category $\mathsf{Sob}$ of sober topological spaces. Here $\Omega$ maps a space to its lattice of open subsets, and $\pt$ maps a spatial frame $\catF$ to its set of points together with the open subsets given by $U_x := \{p \colon \catF \to \mathbf{2} \mid p(x) = 1\}$ for $x\in \catF$.
    \end{theorem}

    \begin{remark}\label{rem:pre_Stone}
        Really there is an adjunction
        \[
        \begin{tikzcd}
            {{\mathsf{Frm}^{\on{op}}}} \\ {{\mathsf{Top}}\rlap{ ,}}
            \arrow[""{name=0, anchor=center, inner sep=0}, "{\pt}" {yshift=-5pt}, bend left=45, from=1-1, to=2-1]
            \arrow[""{name=1, anchor=center, inner sep=0}, "{\Omega}" {yshift=-5pt}, bend left=45, from=2-1, to=1-1]
            \arrow["\dashv"{anchor=center}, draw=none, from=1, to=0]
        \end{tikzcd}
        \]
        between the category $\mathsf{Frm}$ of frames and category $\mathsf{Top}$ of topological spaces, and \Cref{thm:Stone} comes from restricting those objects on which the counit and unit are isomorphisms.
    \end{remark}

    Recall that a topological space $X$ is called \emph{spectral} if it is quasi-compact, $T_0$, the quasi-compact opens $K^\circ(X)$ form a basis of open subsets and are closed under finite intersections.
    By \cite[Theorem 6]{Hoc69}, every spectral space is homeomorphic to the spectrum of a commutative ring.
    By \cite[Theorem II.3.4]{Joh82} any coherent frame is spatial; in fact by Stone duality coherent frames correspond to spectral spaces, i.e.,\ \Cref{thm:Stone} restricts to the antiequivalence
    \[
    \begin{tikzcd}
        {\mathsf{CohFrm}^{\on{op}}} & {\mathsf{Spectral}}\rlap{,}
        \arrow["{\pt}", shift left=1, from=1-1, to=1-2]
        \arrow["\Omega", shift left=1, from=1-2, to=1-1]
    \end{tikzcd}
    \]
    where the morphisms in $\mathsf{Spectral}$ are those continuous maps for which the inverse image of quasi-compact opens remain quasi-compact. Combining \Cref{prop:dist_lattice=coherent_frame} and \Cref{thm:Stone} one therefore sees that distributive lattices correspond to spectral spaces as well, identifying with the sublattice of quasi-compact opens on the corresponding spectral space.

    For the more algebraically inclined reader, let us make the link between points and primes. Recall an element $1\neq \frakp \in \catL$ is \textit{meet-prime} if $x \wedge y \leq \frakp$ implies $x \leq \frakp$ or $y \leq \frakp$.

    \begin{proposition} \label{prop:frameptsaremeetprimes}
        Let $\catF$ be a frame. Then the space of points $p\colon \catF \to \mathbf{2}$ corresponds to the set of meet-prime elements $\frakp \in \catF$ via the bijection \[\frakp \mapsto \left(p_{\frakp}\colon x\mapsto  \begin{cases}
            1 & x \not\leq \frakp\\
            0 & x \leq \frakp
        \end{cases}\right).\]
        Under this bijection, for all $x \in \catF$, the open set $U_x = \{p \mid p(x) = 1\}\in  \Omega(\pt(\catF))$ corresponds to the set of meet-primes $ \{\frakp \mid x \not\leq \frakp\}=:\supp(x) $.
    \end{proposition}
    \begin{proof}
        It is straightforward to verify the assignment is well-defined and injective. It remains to show every point is of this form. Let $p \in \Hom_{\mathsf{Frm}}(\catF, \mathbf{2})$ be a point.
        As $p$ is a morphism of frames, $p\inv(0)$ is a prime ideal and taking $a:=\vee p\inv(0)$ it follows that $p\inv(0)=\down a$ and thus $a$ is meet-prime.  Furthermore, by construction \[p(x) = \begin{cases}
            1 & x \not\leq a\\ 0 & x\leq a
        \end{cases}\] therefore $a$ is our sought after meet-prime.
    \end{proof}

    \begin{remark}
        We note that the terminology ``support'' of an element in $\catF$ hearkens back to the classical algebro-geometric and topological notions of support. Indeed, given a frame $\catF$ and an element $\fraki\in \catF$, one can define a quotient frame $\catF/{\down\fraki}$ as follows.
        Define a congruence relation of $\catF$ by $a\equiv_\fraki b$ if and only if there exist $i,j\in\down\fraki$ with $a\vee i=b\vee j$.
        Then the quotient frame $\catF/{\down\fraki}$ is the set of equivalence classes and it is readily verified it remains a frame (see e.g., \cite{DP66}).
        It follows that $\supp(x)$ is simply the collection of meet-primes $\frakp$ for which $x$ does not identify with $0$ in $\catF/{\down\frakp}$.
        Topologically, $\supp(x) = U_x$ and $U_x$ is the collection of all points $p$ whose support contains $x$.
    \end{remark}

    \begin{proposition}\label{prop:equivalent_def_of_spatial}
        A frame $\catF$ is spatial if and only if each element of $\catF$ can be expressed as a meet of meet-prime elements.
    \end{proposition}
    \begin{proof}
        If every element is a meet of meet-primes, $x\not\leq y$ implies that there exists a meet-prime $\frakp$ with $x\not\leq\frakp$ and $y\leq\frakp$; glancing at \Cref{prop:frameptsaremeetprimes} we see that the corresponding point separates $x$ and $y$.
        Conversely, if $\catF$ is spatial, suppose $y:=\wedge_{x\leq\frakp}\frakp\not\leq x$. Then there is a meet-prime $\frakp$ with $x\leq\frakp$ and $y\not\leq\frakp$ which contradicts the definition of $y$.
    \end{proof}

    \subsection{Hochster duality}
    In \cite[Proposition 8]{Hoc69} Hochster proved that spectral spaces have a duality theory, now referred to as Hochster duality.

    \begin{definition}
        Let $X$ be a spectral space.
        The \emph{Hochster dual} of $X$, denoted by $X^\vee$, has the same underlying set as $X$ and has topology generated by taking the closed subsets of $X$ with quasi-compact complement as a basis of open subsets.
    \end{definition}

    This is a duality, meaning $(X^\vee)^\vee=X$.
    The open subsets in $X^\vee$ are sometimes referred to as \emph{Thomason subsets} of $X$ (or alternatively \textit{inverse-open} subsets).
    By definition, a Thomason subset is a union of complements of quasi-compact opens of $X$.

    \subsection{The spectrum of a lattice}\label{sec:Spc_BDLat}
    To finish this section we recall the following construction; it has appeared before in e.g.,\ \cite{BKS07,DST19,Kra24}.

    \begin{definition}
        Let $\catL$ be a bounded distributive lattice.
        The \textit{spectrum of $\catL$} is defined as the space of points of its ideal lattice,
        \[
        \Spc(\catL):=\pt(\Id(\catL)).
        \]
    \end{definition}

    The spectrum of a lattice is a spectral space by \Cref{prop:dist_lattice=coherent_frame} and \Cref{thm:Stone}.
    In order to give a, perhaps, more familiar description of the spectrum let us recall prime ideals of a lattice $\catL$: a proper ideal $\catP \subset \catL$ is \textit{prime} if it is meet-prime in the ideal lattice $\Id(\catL)$.
    Equivalently, if $\catI \cap \catJ \subseteq \catP$, for ideals $\catI,\catJ\subseteq \catL$, then $\catI \subseteq \catP$ or $\catJ \subseteq \catP$.
    \begin{lemma}
        Let $\catL$ be a bounded distributive lattice.
        The spectrum $\Spc(\catL)$ identifies with the subset of $\Id(\catL)$ consisting of prime ideals.
        The collection \[\supp(x) := \{\catP \in \Spc(\catL) \mid x\not\in \catP\},\quad x \in \catL,\] forms a basis of \textbf{open} sets for the topology on $\Spc(\catL)$.
    \end{lemma}
    \begin{proof}
        Follows from \Cref{prop:frameptsaremeetprimes}.
    \end{proof}

    We note that from the point of view of Stone duality, it is more natural to specify the supports as open, rather than closed, as one who is familiar with the tensor triangular setting may expect.
    As a warning: this is the opposite convention of the spectrum as defined in \cite{Kra24} (but is the same as in \cite{DST19}).
    We promise the reader this is not an arbitrary choice, and the discrepancy will be justified in the sequel.
    However, in this setup the choice of open versus closed is of minimal importance: one obtains the perhaps more familiar closed basis by applying Hochster duality for spectral spaces.

    \begin{lemma}
        The Hochster dual $\Spc(\catL)^\vee$ of the spectrum, is as a set identical to $\Spc(\catL)$ but has the supports as a \textbf{closed} basis for the topology.
    \end{lemma}
    \begin{proof}
        This is standard, see e.g. \cite[Section 4]{Kra24}.
    \end{proof}

    One can translate many of the usual properties of primes for commutative rings or tensor triangulated categories into this setting.
    The essential ingredient is prime lifting.

    \begin{lemma}[Prime Lifting]\label{lem:primeliftinglattice}
        Let $F$ be a filter (that is, $F$ is an nonempty upward set and $x, y \in F$ implies $x \wedge y \in F$) in a bounded distributive lattice.
        Suppose $\catI$ is an ideal maximal among those disjoint from $F$. Then $\catI$ is prime.
    \end{lemma}
    \begin{proof}
        This is \cite[Theorem I.2.4]{Joh82}.
    \end{proof}

    As in the tensor triangulated setting, the supports (or complements of supports) of elements are precisely the quasi-compact opens. This is analogous to \cite[Lemma 2.13, Proposition 2.14]{Bal05}.

    \begin{proposition}[{\cite[3.1.7]{DST19}}]\label{prop:quasicompacts}
        Let $\catL$ be a bounded distributive lattice.
        We have equalities
        \[
        K^\circ(\Spc(\catL)^\vee) = \{\supp(x)^c\mid x \in \catL\}\text{ and }K^\circ(\Spc(\catL)) = \{\supp(x) \mid x\in\catL\}.
        \]
        In particular, the Thomason subsets of $\Spc(\catL)^\vee$, equivalently the open subsets of $\Spc(\catL)$, are unions of supports.
    \end{proposition}

    \section{Support theory for coherent frames}\label{sec:support}

    In this section we spell out support theory for coherent frames, largely following and building on \cite{BKS07, Kra24} (see also \cite[Chapter 3]{DST19}). We give some criteria for these to be ``classifying''; in particular, we provide an analogue of \cite[Theorem 3.1]{Del10} and  \cite[Theorem 6.2.1]{NVY22}.
    As support theory can be viewed as a translation of use Stone duality, it is more natural to work with \emph{open} support data instead of closed; we briefly remark on the closed support data at the end, see \Cref{rmk:closed_supp}.

    As a coherent frame is uniquely determined by its (bounded distributive) sublattice of compact elements, one can rephrase their support theory entirely in terms of the latter.
    The following notion of (classifying) support data for a bounded distributive lattice coincides with the abstract (classifying) support datum on an ideal lattice from \cite{BKS07}, which generalizes the classical tensor triangular notion of \cite{Bal05}.

    \begin{definition}
        Let $\catL$ be a bounded distributive lattice.
        \begin{enumerate}
            \item An \textit{open support datum} on $\catL$ is a tuple $(X, \sigma)$ consisting of spectral space $X$ and a (bounded) lattice morphism $\sigma\colon \catL \to \Omega(X)$. Explicitly $\sigma$ satisfies the following properties:
            \begin{enumerate}
                \item $\sigma(0) = \emptyset$, $\sigma(1) = X$,
                \item $\sigma(x \vee y) = \sigma(x) \cup \sigma(y)$,
                \item $\sigma(x \wedge y) = \sigma(x) \cap \sigma(y)$.
            \end{enumerate}
            \item A morphism of support data $(X, \sigma) \to (Y, \tau)$ is a continuous map $f\colon  X \to Y$ such that $\sigma = f\inv\circ \tau$.
            \item An open support datum $(X, \sigma)$ is \textit{classifying} if we have an order-preserving bijection $\Id(\catL) \to \Omega(X)$ induced by the assignments $\catI \mapsto \cup_{x \in \catI} \sigma(x)$, with inverse induced by the assignment $U \mapsto \{x \in \catL \mid \sigma(x) \subseteq U\}$.
        \end{enumerate}
    \end{definition}
    \begin{remark}
        Note that a morphism of support data is an isomorphism if and only if the underlying map of topological spaces is a homeomorphism.
    \end{remark}

    The spectrum of a lattice is \emph{the} example of a classifying support data.
    The following lemma gives lattice-theoretic versions of key results in \cite{Bal05, BKS07}; note that \Cref{item:cs1} is \cite[Theorem 7 \& Corollary 10]{Kra24}.

    \begin{proposition}\label{prop:coherent_support}
        Let $\catL$ be a bounded distributive lattice.
        \begin{enumerate}
            \item\label{item:cs1} The pair $(\Spc(\catL) , \supp)$ is classifying and is the final open support datum on $\catL$.
            \item\label{item:cs2} Let $f\colon (X, \sigma) \to (Y,\tau)$ be a morphism of open support data on $\catL$.
            If both support data are classifying, then $f\colon  X \to Y$ is a homeomorphism.
            In particular, an open support datum $(X, \sigma)$ is classifying if and only if the canonical morphism $(X, \sigma) \to (\Spc(\catL), \supp)$ is an isomorphism.
        \end{enumerate}
    \end{proposition}
    \begin{proof}
        \ref{item:cs1}:
        That $(\Spc(\catL) , \supp)$ is final comes from the adjunction (\Cref{rem:pre_Stone})
        \[
        \mathsf{Frm}(\id(\catL),\Omega(X))\cong \mathsf{Top}(X, \pt(\id(\catL)),
        \]
        together with the observation that bounded lattice morphisms $\catL \to \Omega(X)$ are in bijection with frame morphisms $\id(\catL) \to \Omega(X)$.
        The classifying part follows as $\supp\colon \id(\catL)\to \Omega(\pt(\id(\catL)))$ is an isomorphism.

        \ref{item:cs2}: By definition of a support datum being classifying we have
        \[
        \Omega(Y)\cong \id(\catL)\cong \Omega(X)
        \]
        and this is simply given by $V\mapsto f\inv(V)$ since $\sigma = f\inv\circ \tau$.
        As both $X$ and $Y$ are spectral, it follows that $f$ is a homeomorphism.
    \end{proof}

    Next, we give a lattice-theoretic analogue of \cite[Theorem 3.1]{Del10} and \cite[Theorem 6.2.1]{NVY22} - in particular, we provide conditions for a support datum to be classifying, and hence homeomorphic to the spectrum.

    \begin{definition}\label{def:supportdatum}
        We say an open support datum $(X, \sigma)$ on a bounded distributive lattice $\catL$ satisfies
        \begin{enumerate}
            \item \textit{injectivity} if $\sigma$ is an injective map of sets,
            \item \textit{faithfulness} if $\sigma(x) = \emptyset$ implies $x = 0$,
            \item \textit{realization} if the quasi-compact opens subsets $K^\circ(X)$ lie in the image of $\sigma$,
            \item \textit{Noetherian realization} if $\sigma$ is surjective.
        \end{enumerate}
    \end{definition}

    \begin{proposition}\label{prop:classifying}
        An open support datum $(X, \sigma)$ for a bounded distributive lattice $\catL$ is classifying if and only if it injective and realizing.
        In particular, $(\Spc(\catL), \supp)$ (uniquely) satisfies injectivity (hence faithfulness) and realization.
    \end{proposition}
    \begin{proof}
        Consider the order-preserving mapping
        \begin{equation}\label{eq:supp_class_mapping}
            \Id(\catL) \to \Omega(X),\ \catI \mapsto \bigcup_{x \in \catI} \sigma(x).
        \end{equation}
        By definition $(X, \sigma)$ is classifying if and only if this map is a bijection.
        But as both sides are coherent frames, \Cref{eq:supp_class_mapping} is a bijection if and only if it induces an isomorphisms on the compact parts, i.e.,\ $\sigma$ induces a bijection
        \[
        \catL \to K^{\circ}(X),\ x \mapsto \sigma(x).
        \]
        Clearly this is equivalent to $\sigma$ being injective and realizing.
        The last statement follows from \Cref{prop:coherent_support}.
    \end{proof}
    \begin{remark}\label{rmk:big_small}
        We require injectivity over faithfulness in contrast to the analogous tensor triangulated geometric results \cite[(5.1.1)]{NVY22} or \cite[Theorem 3.1(S9)]{Del10}. Indeed, it is evident from the proof that we genuinely require $\sigma$ to be an injective homomorphism.
        However, the weaker condition of $\sigma(x) = \emptyset$ implies $x = 0$ does not imply $\sigma$ is an injective lattice homomorphism without additional assumptions. For an easy example, consider the three element poset $\mathbf{3}$ and the homomorphism given by the assignment $0 \mapsto 0$ and $1, 2\mapsto 2$.

        The key point, which is certainly known to the experts, is that in tensor triangular geometry one needs extended support data on `big' categories, where the tensor idempotents and (co)localization sequences live, to show faithfulness implies injectivity on the lattice; see \cite[Lemma 3.5]{Del10} and \cite[Theorem 5.3.1]{NVY22}.
    \end{remark}

    Noetherian realization is strong as it implies $\Spc(\catL)$ is Noetherian.
    This is an analogue of the assumptions in \cite[Theorem 6.2.1]{NVY22}.

    \begin{corollary}\label{cor:noetherianrealizing}
        Suppose an open support datum $(X,\sigma)$ for a bounded distributive lattice $\catL$ is injective and Noetherian realizing. Then $X$ is Noetherian and $(X,\sigma)$ is a classifying support datum.
    \end{corollary}
    \begin{proof}
        As Noetherian realizing is stronger than realizing, $(X,\sigma)$ is classifying. To see Noetherianity of $X$ use \cite[Proposition 7.13]{BF11} and the fact that by Noetherian realization, every open is quasi-compact.
    \end{proof}

    The rest of the results in this section are not used in the sequel, but are noted for interest's sake. We can rephrase \Cref{prop:classifying} in terms of the universal map into the spectrum.
    Observe that by Stone duality, a continuous map of spectral topological spaces $f\colon  X \to Y$ is a monomorphism (resp.\ epimorphism) if and only if the induced map on the coherent lattices of opens $f^{-1}\colon \Omega(Y)\to \Omega(X)$ is an epimorphism (resp.\ a monomorphism).
    Moreover, a map of spectral spaces $f$ is a monomorphism (resp.\ epimorphism) if and only if it is a continuous injective (resp.\ surjective) map if and only if  $f^{-1}$ is surjective (resp.\ injective), see e.g.,\ \cite[Theorems 5.2.2 \& 5.2.5]{DST19}.
    Note, however, that epimorphisms of bounded distributive lattices are not necessarily surjective, although any surjective lattice homomorphism is an epimorphism (see \cite{Mag18} for an example of a non-surjective epimorphism). Therefore one cannot quite state the following purely in terms of the bounded distributive lattice using \Cref{prop:dist_lattice=coherent_frame}.

    \begin{proposition}
        Let $(X,\sigma)$ be an open support datum for a bounded distributive lattice $\catL$.
        Then,
        \begin{enumerate}
            \item the universal map $(X,\sigma) \to (\Spc(\catL),\supp)$ is surjective if and only if $(X,\sigma)$ is injective, i.e., $\sigma\colon \catL\to\Omega(X)$ is injective.
            \item the universal map $(X,\sigma) \to (\Spc(\catL),\supp)$ is injective if and only if the map $\Id(\catL)\to\Omega(X)$ induced by $\sigma$ is surjective.
            In particular, the latter is the case if $(X,\sigma)$ is realizing.
        \end{enumerate}
    \end{proposition}
    \begin{proof}
        If $f\colon (X,\sigma) \to (\Spc(\catL), \supp)$ is the support data induced by $\sigma$, then
        $f^{-1}\colon \Id(\catL)\allowbreak=\Omega(\Spc(\catL))\to \Omega(X)$ is simply the map induced by $\sigma$.
        The statements then follow from the preceding discussion.
    \end{proof}

    Given \Cref{rmk:big_small}, it is useful to note injectivity is equivalent to the key property of order-reflection shown in \cite[Lemma 3.3]{BCR97}, \cite[Lemma 3.14]{Tho97}, and \cite[Lemma 4.1.9]{KP17}.

    \begin{proposition}
        Suppose $(X, \sigma)$ is an open support datum for a bounded distributive lattice $\catL$. Then $(X,\sigma)$ is injective if and only if $\sigma$ is order-reflecting, i.e., the following holds: for any $x,y \in \catL$, $x \leq y$ if and only if $\sigma(x) \subseteq \sigma(y)$.
    \end{proposition}
    \begin{proof}
        Note one implication is automatic as $\sigma$ is order preserving.
        (E.g.,\ $x \leq y$ is equivalent to $x \vee y = y$, so $\sigma(y) = \sigma(x \vee y) = \sigma(x) \cup \sigma(y)$ and $\sigma(x) \subseteq \sigma(y)$.)

        First, suppose $(X,\sigma)$ is injective. If $x \not\leq y$, then $y \neq x\vee y$.
        Consequently, by assumption, $\sigma(y)\neq\sigma(x\vee y) =\sigma(x) \cup \sigma(y)$ and hence $\sigma(x) \not\subseteq \sigma(y)$.
        Conversely, if $(X,\sigma)$ is not injective, there are $x\neq y$ with $\sigma(x) = \sigma(y)$.
        Clearly, $x\not\leq y$ or $y\not\leq x$ showing that $\sigma$ does not reflect the order.
    \end{proof}

    \begin{remark}\label{rmk:closed_supp}
        In the abstract lattice-theoretic setting we consider in this section, it is natural to speak about open support data, as this is more in line with the language of Stone duality.
        One can still speak of closed support data, changing the lattice of open subsets by the lattice of closed subsets in the definition. With this view, one can make sense of $\Spc(\catL)^\vee$ as the universal closed support datum, as in \cite[Theorem 5]{Kra24}, and state similar detection theorems for classifying support data along the lines of \Cref{prop:classifying}.

        A closed support datum being classifying, as for example in \cite{BKS07}, means that the closed support data yields an isomorphism $\Id(\catL)\cong\Omega(X^\vee)$, i.e.,\ an isomorphism with the Thomason subsets of $X$. This requires that the closed support datum factors through the sublattice of closed subsets with quasi-compact complement, in which case, through Hochster duality, one can translate everything into open support data.

    \end{remark}

    \section{A lattice-theoretic Cohen's theorem}\label{sec:Cohen}

    In \cite{Bar25} Barthel proved a tensor triangular version of Cohen's theorem \cite{Coh50}, characterizing those tt-categories with weakly Noetherian spectrum whose spectrum is finite: these are exactly those for which every radical $\otimes$-ideal is finitely generated. We extend this result here.

    \begin{definition}
        Given a spectral space $X$, we say $X$ is \textit{inverse-}\texttt{P}, for any property \texttt{P}, if $X^\vee$ satisfies property \texttt{P}. For instance, if $X^\vee$ is Noetherian, we say $X$ is \textit{inverse-Noetherian} (note $X$ need not be Noetherian).
    \end{definition}

    \begin{definition}[{\cite[Definition 2.3]{BHS23}}]
        Let $X$ be a spectral space.
        \begin{enumerate}
            \item A subset $V \subseteq X$ is \textit{weakly visible} if it is the intersection of a Thomason subset and the complement of a Thomason subset (i.e., the intersection of an open set and a closed set in $X^\vee$).
            \item A point $x \in X$ is \textit{weakly visible} if the set $\{x\}$ is and $X$ is \textit{weakly Noetherian} if all its points are.
        \end{enumerate}
    \end{definition}
    \begin{remark}
        Again, $X$ is spectral as in the previous definition.
        \begin{enumerate}
            \item Noetherian spectral spaces and $T_1$ (equivalently Hausdorff, equivalently profinite) spectral spaces are weakly Noetherian, see \cite[Remark 2.4]{BHS23}.
            \item By \cite[Lemma 3]{Bar25}, $X$ is finite if and only if $X$ is weakly Noetherian and inverse-Noetherian.
        \end{enumerate}
    \end{remark}

    We translate Barthel's results to the lattice-theoretic setting. The proofs are the same as in \cite{Bar25}, replacing the tensor triangulated category $\catK$ with a bounded distributive lattice $\catL$, the lattice of radical thick $\otimes$-ideals $T_r(\catK)$ with the ideal lattice $\Id(\catL)$, $\Spc(\catK)$ with $\Spc(\catL)^\vee$, and finitely generated ideals with compact elements of $\Id(\catL)$, i.e., principal ideals (by \Cref{prop:dist_lattice=coherent_frame}).

    \begin{proposition}[{\cite[Proposition 6]{Bar25}}]
        For a bounded distributive lattice $\catL$, the following are equivalent:
        \begin{enumerate}
            \item $\Spc(\catL)^\vee$ is inverse-Noetherian, i.e.\ $\Spc(\catL)$ is Noetherian;
            \item every ideal of $\catL$ is principal;
            \item every prime ideal of $\catL$ is principal;
        \end{enumerate}
    \end{proposition}

    \begin{corollary}[{\cite[Theorem 7]{Bar25}}]\label{thm:barthelthm}
        Let $\catL$ be a bounded distributive lattice. The following are equivalent:
        \begin{enumerate}
            \item $\Spc(\catL)^\vee$ is weakly Noetherian and every ideal in $\catL$ is principal;
            \item $\Spc(\catL)^\vee$ is weakly Noetherian and every prime ideal in $\catL$ is principal;
            \item $\Spc(\catL)^\vee$ is finite.
        \end{enumerate}
    \end{corollary}

    The weak Noetherianity condition for $\Spc(\catL)^\vee$ is necessary: as Barthel notes, the stable homotopy category of finite $p$-local spectra $\on{Sp}^{\text{fin}}_{(p)}$ gives a counterexample (see \Cref{def:objlattice} for how one translates this to a bounded distributive lattice).

    \section{Noncommutative tensor triangular geometry}\label{sec:nctt}

    We shift our focus to noncommutative tensor triangular geometry, applying our lattice-theoretic machinery towards the classification of semiprime thick $\otimes$-ideals, i.e., the noncommutative analogue of Balmer's classification of radical thick $\otimes$-ideals in tensor triangular geometry \cite[Theorem 4.10]{Bal05}.

    \subsection{Recollections}

    \begin{notation}
        Throughout, we assume $\catK$ is an \textit{essentially small monoidal tri\-an\-gu\-lat\-ed category}, in particular the tensor product is exact but not necessarily symmetric or braided.
    \end{notation}

    See, e.g., \cite{EGNO15} for more on monoidal categories and \cite{Ne01,HJR11} for more on triangulated categories.

    \begin{definition}[\cite{NVY22}]
        Let $\catK$ be a monoidal triangulated category.
        \begin{enumerate}
            \item A \textit{thick subcategory} of $\catK$ is a full triangulated subcategory of $\catK$ closed under direct summands. A thick subcategory $\catI \subseteq \catK$ is a \textit{thick $\otimes$-ideal} if in addition, for all $x \in \catI$ and $y \in \catK$, $x \otimes y \in \catK$ and $y \otimes x \in \catK$.

            \item A proper thick $\otimes$-ideal $\catP \subset \catK$ is \textit{prime} if for all thick $\otimes$-ideals $\catI, \catJ$ of $\catK$, $\catI \otimes \catJ \subseteq \catP$ implies $\catI \subseteq \catP$ or $\catJ \subseteq \catP$. Equivalently $\catP$ is prime if for all $x, y \in\catK$, $x \otimes \catK \otimes y \subseteq \catP$ implies $x \in \catK$ or $y \in \catK$. Finally, $\catP$ is \textit{completely prime} if for all $x, y \in\catK$, $x \otimes y \in \catP$ implies $x\in\catP$ or $y \in\catP$. If $\catK$ is symmetric, every prime ideal is completely prime.

            \item A thick $\otimes$-ideal $\catI$ is \textit{semiprime} if it is an intersection of prime thick $\otimes$-ideals (with the convention that $\catK$ is the empty intersection). Equivalently, see \cite[Theorem 1.2.1]{NVY22}, $\catI$ is semiprime if for all $x\in\catK$, $x \otimes \catK \otimes x \in \catI$ implies $x \in \catI$, or for any thick $\otimes$-ideals $\catJ$ of $\catK$, $\catJ \otimes \catJ \subseteq \catI$ implies $\catJ \subseteq\catI$.

            \item A thick $\otimes$-ideal $\catI$ is \textit{radical} if for all $x \in \catK$ and $n \in \bbN^+$, $x\in\catI$ if and only if $x^{\otimes n} \in \catI$. When $\catK$ is symmetric, radical equals semiprime.
        \end{enumerate}
        We write $T(\catK)$ (resp.\ $T_{s}(\catK)$ and $T_{r}(\catK)$) to denote the lattice of thick $\otimes$-ideals (resp. semiprime and radical thick $\otimes$-ideals). Given an element $x \in \catK$, we write $\langle x\rangle$ (resp.\ $\langle x\rangle_s$ and $\langle x\rangle_r$) to denote the thick $\otimes$-ideal (resp.\ semiprime and radical thick $\otimes$-ideal) generated by $x$; more explicit descriptions are given in \Cref{rem:concrete_hulls} below.
    \end{definition}

    \begin{remark}\label{rem:concrete_hulls}
        For any $x\in \catK$, $\langle x\rangle$ (resp.\ $\langle x\rangle_s$ and $\langle x\rangle_r$) is the intersection of all ideals (resp.\ prime ideals and completely prime ideals) containing $x$.
        Moreover, as will be used below, $\langle x\rangle$ admits an inductive description as the collection of all objects which can be built from $x$ in finitely many steps by taking (co)cones, shifts, direct summands, and tensoring by objects of $\catK$ on the left or right (see e.g.\ \cite[Definition 3.1]{MR23}). Further, $\langle x\rangle_r$ is formed from $\langle x\rangle$ by adding in radicals, i.e., objects $y$ such that $y^{\otimes n} \in \langle x\rangle$ for some $n \in \bbN$.
        However, an inductive definition does not exist for $\langle x\rangle_s$ without additional assumptions on the monoidal triangulated category $\catK$, which can causes additional difficulty for noncommutative settings.
    \end{remark}

    \begin{definition}
        The \textit{noncommutative Balmer spectrum} $\Spc(\catK)$ of a monoidal triangulated category $\catK$ is, as a set, the collection of prime thick $\otimes$-ideals of $\catK$ equipped with the topology defined by taking the supports of objects $x \in \catK$, \[\supp(x) := \{\catP \in \Spc(\catK) \mid x \not\in \catP\},\] as a basis of \textbf{closed} subsets.
    \end{definition}

    The Balmer spectrum was introduced by Balmer for tensor triangulated categories (i.e.,\ symmetric monoidal triangulated categories) \cite{Bal05}, and extended to the noncommutative setting by Nakano--Vashaw--Yakimov \cite{NVY22}.
    The connection to lattices was first established in Buan--Krause--Solberg \cite{BKS07}, expanded on by Kock--Pitsch \cite{KP17}, and recently revisited by Gratz--Stevenson and Krause \cite{GS23, Kra24}.

    \subsection{\texorpdfstring{$T_s(\catK)$}{Ts(K)} is a spatial frame}

    We will predominantly focus on the lattice of semiprimes. Meets and joins have a simple description, demonstrating that $T_s(\catK)$ is a distributive lattice (in fact, as we shall see below, it is always a spatial frame).
    By \cite[Lemma 6.4.2]{GS23}, for $\catI, \catJ \in T_s(\catK)$, we have $ \catI \wedge \catJ = \catI\cap\catJ = \langle \catI \otimes \catJ \rangle_s  = \langle \catJ \otimes \catI \rangle_s $.
    Moreover, $T_s(\catK)$ is complete with arbitrary meets corresponding to intersections, and arbitrary joins, for $\{\catJ_i\}_i\subseteq T_s(\catK)$, are given as $\vee_i \catJ_i=\langle \cup_i\catJ_i\rangle_s$.

    The following definition was introduced in \cite{GS23}.
    \begin{definition}
        We say a thick $\otimes$-ideal $\catP$ is \textit{prime with respect to semiprimes} if for all \textit{semiprime} $\otimes$-ideals $\catI, \catJ$, $\catI \otimes \catJ \subseteq \catP$ implies $\catI \subseteq \catP$ or $\catJ \subseteq \catP$.
    \end{definition}

    Of course, a prime thick $\otimes$-ideal is prime with respect to semiprimes.
    We claim that, in fact, any semiprime $\otimes$-ideal is prime with respect to semiprimes if and only if it is prime.
    The key is a semiprime analogue of \cite[Lemma 3.1.2]{NVY22}.

    \begin{lemma}\label{lem:semiprimeidealinclusion}
        For every two collections $\catM, \catN \subset \catK$ of objects, we have an inclusion \[\langle \catM \rangle_s  \otimes \langle \catN \rangle_s  \subseteq \langle \catM \otimes \catK \otimes \catN\rangle_s .\]
    \end{lemma}
    \begin{proof}
        This follows essentially the same as the proof of \cite[Lemma 3.1.2]{NVY22} with one additional technicality. First, we show that $\langle \catM \rangle_s  \otimes \catN \subseteq \langle \catM \otimes \catK \otimes \catN\rangle_s $. Let $\catI$ denote the full subcategory of objects $x$ such that for all $n\in \catN$ and $y \in \catK$, $x \otimes y \otimes n \in \langle \catM \otimes \catK \otimes \catN\rangle_s $. Obviously, $\catM \subseteq \catI$.

        We claim that $\catI$ is a semiprime thick $\otimes$-ideal. Verifying $\catI$ is thick, triangulated, and a $\otimes$-ideal follows identically as in \emph{loc.\ cit.} To show $\catI$ is semiprime, let $x \in \catK$ and suppose $x \otimes z \otimes x \in \catI$ for all $z \in \catK$. Then equivalently, $x \otimes z \otimes x \otimes y \otimes n \in \langle \catM \otimes \catK \otimes \catN \rangle_s $ for all $z, y \in \catK$ and $n \in \catN$. In particular, setting $z = y \otimes n\otimes z'$ for any $z' \in \catK$ yields \[(x \otimes y \otimes n) \otimes z' \otimes (x \otimes y \otimes n) \in \langle \catM \otimes \catK \otimes \catN \rangle_s \] for all $y, z' \in \catK$ and $n \in \catN$, therefore semiprimality implies $x \otimes y \otimes n \in \langle \catM \otimes \catK \otimes \catN \rangle_s $, hence $x \in \catI$ as desired. Because $\catM \subseteq \catI$ and $\catI$ is a semiprime thick $\otimes$-ideal, $\langle \catM \rangle_s  \subseteq \catI$ as well. Therefore, $\langle \catM \rangle_s  \otimes \catN \subseteq \langle \catM \otimes \catK \otimes \catN\rangle_s $.

        Then by analogous arguments, we have $\catM \otimes \langle \catN \rangle_s  \subseteq \langle \catM \otimes \catK \otimes \catN \rangle_s $.
        Lastly, replacing $\catN$ by $\langle \catN \rangle_s $ in the previous paragraph and using the last sentence we have $\langle\catM\rangle_s  \otimes \langle \catN \rangle_s  \subseteq \langle \catM \otimes \catK \otimes \catN \rangle_s $, as desired.
    \end{proof}
    \begin{proposition}\label{prop:prime=prime_wrt_semiprime}
        Let $\catP$ be a proper thick $\otimes$-ideal. The following are equivalent:
        \begin{enumerate}
            \item\label{item:p=psp1} $\catP$ is prime, i.e., for all thick $\otimes$-ideals $\catI, \catJ$, $\catI \otimes \catJ \subseteq \catP$ implies $\catI \subseteq \catP$ or $\catJ \subseteq \catP$;
            \item\label{item:p=psp3} For all $x,y \in\catK$, $x \otimes \catK \otimes y \subseteq \catP$ implies $x\in \catP$ or $y \in\catP$.
            \item\label{item:p=psp4} $\catP$ is semiprime and prime with respect to semiprimes, i.e., for all semiprime thick $\otimes$-ideals $\catI, \catJ$, $\catI \otimes \catJ \subseteq \catP$ implies $\catI \subseteq \catP$ or $\catJ \subseteq \catP$;
            \item\label{item:p=psp5} $\catP$ is meet-prime in $T_s(\catK)$.
        \end{enumerate}
    \end{proposition}
    \begin{proof}
        \ref{item:p=psp1} $\Leftrightarrow$ \ref{item:p=psp3} and \ref{item:p=psp4} $\Leftrightarrow$ \ref{item:p=psp5} are \cite[Theorem 3.2.2]{NVY22} and \cite[Lemma 6.4.3]{GS23} respectively.
        Evidently \ref{item:p=psp1} implies \ref{item:p=psp4}, so it rests to show \ref{item:p=psp4} implies \ref{item:p=psp3}.
        Suppose $\catP$ is prime and $x \otimes \catK \otimes y \subseteq \catP$. Since $\catP$ is semiprime we have using \Cref{lem:semiprimeidealinclusion} an inclusion \[\langle x \rangle_s  \otimes \langle y \rangle_s  \subseteq \langle x\otimes \catK \otimes y \rangle_s  \subseteq \catP,\] and primality implies $\langle x\rangle_s  \subseteq \catP$ or $\langle y \rangle_s  \subseteq \catP$. Therefore $x \in \catP$ or $y \in \catP$, as desired.
    \end{proof}
    \begin{remark}
        \Cref{prop:prime=prime_wrt_semiprime} surprisingly does not need a good description of the semiprime radical, cf.\ \cite[Remark 6.4.4]{GS23}.
    \end{remark}

    \begin{corollary}\label{cor:thicksarespatialframe}
        Let $\catK$ be a monoidal triangulated category. The lattice $T_s(\catK)$ of semiprime thick $\otimes$-ideals is a spatial frame.
    \end{corollary}
    \begin{proof}
        We first verify $T_s(\catK)$ is a frame; for this we must verify for all semiprime thick $\otimes$-ideals, $\catI, \catJ_i$ for $i \in I$, that
        \begin{equation}\label{eq:in_thicksarespatialframe}
            \catI \cap \Big\langle\bigcup_i \catJ_i\Big\rangle_s = \Big\langle\bigcup_i (\catI \cap \catJ_i) \Big\rangle_s.
        \end{equation}
        The left-hand side is identically the intersection of all prime thick $\otimes$-ideals containing either $\catI$ or containing all $\catJ_i$, and the right-hand side is the intersection of all prime thick $\otimes$-ideals containing, for each $i \in I$, either $\catI$ or $\catJ_i$.
        Clearly, if a prime contains either $\catI$ or all $\catJ_i$, it contains for each $i\in I$ either $\catI$ or $\catJ_i$.
        Conversely, if a prime contains for each $i \in I$ either $\catI$ or $\catJ_i$; then either if contains $\catI$ or it has to contain  $\catJ_i$ for all $i$.
        Therefore the equality in \Cref{eq:in_thicksarespatialframe} holds. Now, $T_s(\catK)$ is spatial from \Cref{prop:equivalent_def_of_spatial} and \Cref{prop:prime=prime_wrt_semiprime}, as every element is an intersection of prime ideals, which are equivalently the meet-primes of $T_s(\catK)$.
    \end{proof}
    \begin{remark}\label{ex:non_semiprimes_ex}
        It is important to stress that the tensor product is doing some heavy lifting.
        The fact that $ \catI \wedge \catJ = \langle \catI \otimes \catJ \rangle_s$ for $\catI, \catJ \in T_s(\catK)$ is crucial, since as a consequence, primality in the monoidal triangular sense and meet-primality in the lattice-theoretic sense are identified.

        This need not happen for the lattice $T(\catK)$; meet-prime elements of $T(\catK)$ in general do not correspond to the prime thick $\otimes$-ideals, even in the symmetric case. It is a straightforward exercise to show that primes are meet-prime, but the converse need not hold.
        To give an example, let $k$ be a field and $\mathsf{grmod}_\bbN(k)$ denote the semisimple symmetric monoidal category of finitely generated $\bbN$-graded $k$-modules.
        The simple objects are exactly the twists $k(i)$, i.e.\ the graded modules with $k$ in degree $i$ and zero elsewhere.
        We have $k(i) \otimes k(j) \cong k(i + j)$, in particular $k:=k(0)$ is the tensor unit.
        The thick $\otimes$-ideals of $D^b(\mathsf{grmod}_\bbN(k))$ form a countable descending chain $D^b(\mathsf{grmod}_\bbN(k)) = \langle k(0)\rangle \supset \langle k(1) \rangle \supset \cdots \supset \{0\}$, where we regard objects of $\mathsf{grmod}_\bbN(k)$ as chain complexes concentrated in homological degree 0.
        There are two prime thick $\otimes$-ideals, $\langle k(1) \rangle$ and $\{0\}$, but because the thick $\otimes$-ideals form a chain, every thick $\otimes$-ideal is meet-prime in $T(D^b(\mathsf{grmod}_\bbN(k)))$. This also is an example of a non-rigid tensor triangulated category which has significantly more thick $\otimes$-ideals than radical thick $\otimes$-ideals.
    \end{remark}

    \begin{theorem}\label{cor:classification_for_spatial_frames}
        Let $\catK$ be a monoidal triangulated category. We have a natural isomorphism of frames
        \[T_s(\catK) \isoto \Omega(\Spc(\catK)^\nu),~\catI\mapsto \bigcup_{x\in \catI}\supp(x),
        \] where $\Spc(\catK)^\nu$ denotes the ``pseudo-Hochster dual'' of $(\Spc(\catK),\supp)$. That is, $\Spc(\catK)^\nu = \Spc(\catK)$ as a set with supports $\supp(x)$, $x\in \catK$, forming an \textbf{open} base for the topology.
    \end{theorem}
    \begin{proof}
        \Cref{prop:frameptsaremeetprimes} asserts $\pt(T_s(\catK))$ identifies with the meet-primes of $T_s(\catK)$ with the supports forming an open base for the topology, and \Cref{prop:prime=prime_wrt_semiprime} asserts the meet-primes are precisely the prime thick $\otimes$-ideals of $\catK$.
        Therefore, $\pt(T_s(\catK))=\Spc(\catK)^\nu$.
        The statement then reduces to Stone duality, as $T_s(\catK)$ is spatial by \Cref{cor:thicksarespatialframe}. Therefore, we have natural isomorphisms $T_s(\catK) \cong \Omega(\pt(T_s(\catK)))=\Omega(\Spc(\catK)^\nu)$, where the latter isomorphism is given by \Cref{prop:frameptsaremeetprimes}.
    \end{proof}

    \begin{remark}
        In \Cref{cor:classification_for_spatial_frames}, the space $\Spc(\catK)^\nu$ is not the genuine Hochster dual of $\Spc(\catK)$, since $\Spc(\catK)$ need not be a spectral space.
        However, one can still construct $\Spc(\catK)$ from $\Spc(\catK)^\nu$ and vice versa, but only when one also has the support data $\supp$ for $\catK$; for one these form a basis of closed sets and for the other a basis of open sets - in this sense the datum $(\Spc(\catK),\supp)$ and $(\Spc(\catK)^\nu, \supp)$ are ``dual,'' hence the terminology ``pseudo-Hochster dual.'' On the other hand, if $\Spc(\catK)$ is spectral, then $\Spc(\catK)^\vee = \Spc(\catK)^\nu$; in this case we write $\Spc(\catK)^\vee$.
    \end{remark}

    We end this subsection by noting that not every sober space can arise in the above fashion.

    \begin{proposition}\label{prop:quasicompactfullspace}
        Let $\catK$ be a monoidal triangulated category. Then $\catK \in T_s(\catK)$ is a compact element. Equivalently, $\Spc(\catK)^\nu$ is quasi-compact.
    \end{proposition}
    \begin{proof}
        Let $S \subseteq T_s(\catK)$ be a collection of semiprime thick $\otimes$-ideals satisfying \[\left\langle \bigcup_{\catI \in S} \catI\right\rangle_s = \bigvee_{\catI \in S} \catI  = \catK.\] Note $\vee_{\catI \in S} \catI$ is equivalently the intersection of all prime thick $\otimes$-ideals containing every $\catI \in S$, and by assumption, this intersection is empty, so no prime contains every element of $S$. Now, if the (not necessarily semiprime) thick $\otimes$-ideal $\langle \cup S\rangle$ did not contain $\bbone\in \catK$, then prime lifting \cite[Lemma A.1.1]{NVY23} asserts that there exists a prime thick $\otimes$-ideal $\catP \in \Spc(\catK)$ satisfying $\langle \cup_{\catI \in S} \catI \rangle\subseteq \catP$, a contradiction. Therefore, $\bbone \in \langle \cup_{\catI \in S} \catI\rangle$. Because every element of $\langle \cup_{\catI \in S} \catI\rangle$ can be built in finitely many steps by tensoring on the left or right, shifting, and taking direct summands and cones, it follows that there exists a finite subset $S' \subseteq S$ such that $\bbone \in \langle \cup_{\catI \in S'} \catI\rangle$, and thus $\langle \cup_{\catI \in S'} \catI\rangle_s = \catK$, as desired. The final statement follows from \Cref{cor:classification_for_spatial_frames} since, under Stone duality, compact elements of a spatial frame $\catF$ correspond bijectively to the quasi-compact opens of $\pt(\catF)$.
    \end{proof}

    \begin{remark}
        By \cite[Theorem 6]{Hoc69} every spectral space is homeomorphic to the spectrum of a commutative ring.
        Given that $\Spc(D^{\on{perf}}(R)) \cong \Spec(R)$ for any commutative ring $R$, by \cite{Tho97, Bal05}, it follows that every spectral space arises as the Balmer spectrum of a tensor triangulated category.
        Equivalently, by Stone duality, every coherent frame arises as the frame of radical thick $\otimes$-ideals of a tensor triangulated category $\catK$.

        In contrast, by \Cref{prop:quasicompactfullspace}, not every sober topological space arises as the pseudo-Hochster dual of the Balmer spectrum of a monoidal triangulated category, as sober topological spaces need not be quasi-compact. For instance, $\bbR^n$ is sober, since any Hausdorff topological space is sober.
        Equivalently, by Stone duality, not every spatial frame arises as the frame of semiprime thick $\otimes$-ideals of a monoidal triangulated category. It would be interesting to know what additional restrictions one needs to impose on a sober space or spatial frame to guarantee that it arises from a monoidal triangulated category.
    \end{remark}

    \subsection{On group actions on monoidal triangulated categories}\label{sec:groupactions1}

    To show that in general $T_s(\catK)$ need not be a coherent frame, we briefly remark on group actions on monoidal triangulated categories, the corresponding crossed product categories, and Huang-Vashaw's classification result \cite[Proposition 7.4]{HV25}. We note that such categories encompass stable module categories of the Hopf algebras considered by Benson--Witherspoon \cite{BW14} and Bergh--Plavnik--Witherspoon \cite{BPW24} which demonstrate particularly noncommutative behavior, such as the tensor product property not holding for $\Spc(\catK)$.

    Here, semiprime thick $\otimes$-ideals are parametrized by specialization-closed subsets, which can be a proper superset of Thomason subsets, as \cite[Example 7.7]{HV25} shows. We briefly review the relevant constructions and results.

    \begin{definition}
        Let $G$ be a group acting on a monoidal triangulated category $\catK$ via monoidal triangulated autoequivalences.
        \begin{enumerate}
            \item The \textit{crossed product category} of $\catK$ by $G$ is defined to be the direct sum \[\catK \rtimes G := \bigoplus_{g \in G}\catK\] indexed by elements of $G$. For any element $x \in \catK$, the corresponding object in $\catK \rtimes G$ indexed by $g \in G$ is denoted $x \boxtimes g$. By definition \[\Hom_{\catK \rtimes G}(x \boxtimes g, y \boxtimes h) := \begin{cases}
                \Hom_\catK(x,y) & g = h,\\ 0 & g \neq h.
            \end{cases}\]
            It is endowed with a monoidal structure given by the formula \[(x \boxtimes g) \otimes (y \boxtimes h) := \left(x\otimes g(y)\right) \boxtimes gh.\]
            \item A thick subcategory of $\catK$ is a \emph{$G$-ideal} if it is a $G$-stable $\otimes$-ideal; a proper $G$-ideal is a \emph{$G$-prime $\otimes$-ideal} if for all $G$-ideals $\catI, \catJ$, $\catI \otimes \catJ \subseteq \catP$ implies $\catI \subseteq \catP$ or $\catJ \subseteq \catP$. We warn the reader that a $G$-prime $\otimes$-ideal may not necessarily be a prime $\otimes$-ideal!
            \item Let $G$-$\Spc(\catK)$ denote the \emph{$G$-Balmer spectrum}, the set of $G$-prime $\otimes$-ideals equipped with the topology defined by taking the $G$-Balmer supports of objects $x \in \catK$, \[\supp_G(x) = \{\catP \in G\on{-}\Spc(\catK) \mid x\not\in \catP\},\] as a basis of closed subsets.
        \end{enumerate}
    \end{definition}

    \begin{proposition}[{\cite[Proposition 5.4]{HV25}}]
        Let $G$ be a group acting on a monoidal triangulated category $\catK$.
        \begin{enumerate}
            \item The assignment  $\catI\mapsto\catI \rtimes G:=\operatorname{add}\{ i\boxtimes g \mid i\in\catI,g\in G\} \subseteq \catK \rtimes G$ is an order-preserving bijection between the $G$-ideals of $\catK$ and the thick $\otimes$-ideals of $\catK \rtimes G$.
            \item The bijection restricts to a homeomorphism $G$-$\Spc(\catK) \cong \Spc(\catK \rtimes G)$.
        \end{enumerate}
    \end{proposition}

    We say a monoidal triangulated category is \emph{half-rigid} if every object has either a left or right dual.

    \begin{proposition}[{\cite[Proposition 7.4]{HV25}}]
        Let $G$ be a group acting on a half-rigid monoidal triangulated category $\catK$ with Noetherian spectrum. The $G$-Balmer support on $G$-$\Spc(\catK)$ induces a bijection between $G$-ideals of $\catK$ and specialization-closed subsets of $G$-$\Spc(\catK)$.
    \end{proposition}

    \begin{observation}\label{obs:Ts_can_be_ugly}
        If $\Spc(\catK)$ is Noetherian, then every specialization-closed subset of $\Spc(\catK)$ is a union of supports (cf.\ \cite[Remark 3.6]{Row24}).
        When $\catK$ is half-rigid the same holds true for
        $\Spc(\catK \rtimes G)$; even though the latter need not be Noetherian again.
        Indeed, a specialization-closed subset of $\Spc(\catK \rtimes G)$ is a union of closed subsets, and by the proof of \cite[Proposition 7.4]{HV25} any closed subset is a union of supports.
        It follows that the specialization-closed subsets of $\Spc(\catK \rtimes G)$ correspond bijectively to open subsets of $\Spc(\catK \rtimes G)^\nu$.
        Consequently, the bijection given in \cite[Proposition 7.4]{HV25} is the bijection given by \Cref{cor:classification_for_spatial_frames} applied to $\catK \rtimes G$.

        Furthermore, in \cite[Example 7.7]{HV25} the authors give an example of a monoidal triangulated category $\catK$ with non-spectral Balmer spectrum, as the spectrum itself is not quasi-compact; hence, an example of a category $\catK$ with $T_s(\catK)$ a non-coherent frame.
        From what we show in \Cref{thm:distlattice} below (see also \Cref{rmk:examplesofclassification} for a discussion), because this category is rigid, it must have two principal semiprime ideals whose intersection is not principal; consequently, the sub-poset of compact objects of $T_s(\catK)$ does not form a sublattice. Hence neither $\Spc(\catK)$ nor $\Spc(\catK)^\nu$ are spectral topological spaces.
    \end{observation}

    Unfortunately, it does not seem easy to describe the sets corresponding to the open subsets of $\Spc(\catK \rtimes G)^\nu$ when $\catK$ is no longer Noetherian, since the arguments in \cite{HV25} rely heavily on Noetherianity.
    In general, it would be interesting to know whether there are conditions characterizing when open subsets of $\Spc(\catK)^\nu$ correspond to some type of subsets of $\Spc(\catK)$.

    We will revisit crossed product categories again in the sequel, after more terminology is introduced. Under mild hypotheses, we will see that some crossed product categories still remain well-behaved, see \Cref{sec:crossedproduct}.

    \section{Coherence and the principal part}\label{sec:coh_and_pp}

    We saw in \Cref{cor:classification_for_spatial_frames} that unlike in the symmetric case, it is \emph{not} the Thomason subsets of the Balmer spectrum that classify the semiprime ideals.
    The reason for this is exhibited in \Cref{obs:Ts_can_be_ugly}: $T_s(\catK)$ can in general fail to be a coherent frame.
    However, once it is a coherent frame (and the compact elements are what they should be), one recovers the `usual' classification (cf. \cite[Theorem 4.10]{Bal05}).

    \begin{theorem}\label{thm:balmerclassification_without_conditions}
        Let $\catK$ be a monoidal triangulated category. Then $\Spc(\catK)^\nu$ is a spectral space if and only if $T_s(\catK)$ is coherent.
        In addition, if the supports $\supp(x)$ are exactly the quasi-compact open subsets of $\Spc(\catK)^\nu$, we obtain an order-preserving bijection \[T_s(\catK) \cong \on{Th}(\Spc(\catK))\] between the semiprime thick $\otimes$-ideals of $\catK$ and the Thomason subsets of $\Spc(\catK)$.
    \end{theorem}
    \begin{proof}
        As $\Spc(\catK)^\nu$ is precisely $\pt(T_s(\catK))$, the first assertion follows from Stone duality and  \Cref{cor:classification_for_spatial_frames}.
        For the second assertion, if the supports are exactly the quasi-compact open subsets, then $\Spc(\catK)^\nu$ is the genuine Hochster dual of $\Spc(\catK)$. Indeed, we have $(\Spc(\catK)^\nu)^\vee=\Spc(\catK)$, as the closed basis for $(\Spc(\catK)^\nu)^\vee$ are then exactly the supports, from which is follows that $\Spc(\catK)^\nu=\Spc(\catK)^\vee$. The open subsets of the Hochster dual are exactly the Thomason subsets.
    \end{proof}

    Thus in order to obtain the more common classification in terms of Thomason subsets of the Balmer spectrum and support, we need to know when $T_s(\catK)$ is a coherent frame and when the supports of objects correspond exactly to the quasi-compact opens of $\Spc(\catK)^\nu$.
    By \Cref{cor:classification_for_spatial_frames} the support $\supp(x)$ corresponds to the finitely generated semiprime ideal $\langle x\rangle_s$ and therefore the condition on quasi-compactness translates to asking when the finitely generated semiprime ideals are exactly the compact elements in $T_s(\catK)$. Therefore, we aim to answer the following question:

    \begin{question}
        When is $T_s(\catK)$ a coherent frame, with its sublattice of compact elements corresponding to the principal semiprime thick $\otimes$-ideals $\langle x\rangle_s$?
    \end{question}

    \subsection{The principal part}

    To this end we need to understand the subposet of finitely generated, equivalently principal, semiprime ideals, which we call the principal part.

    \begin{definition}\label{def:objlattice}
        We let $t_s(\catK)$ denote the subposet of $T_s(\catK)$ consisting of principal semiprime $\otimes$-ideals.
        We call $t_s(\catK)$ the \textit{principal part} of $T_s(\catK)$.
    \end{definition}

    In general this is only a join sublattice of $T_s(\catK)$ (see \Cref{lem:distributivityofprincipals} below).
    An alternative way of thinking about this join sublattice is as follows.

    \begin{remark}[cf. {\cite[Section 6]{Kra24}}]
        For $x, y\in\catK$, we write $x \sim y$ if and only if $\langle x \rangle_s  = \langle y \rangle_s $. This is an equivalence relation, and the set of equivalence relations $L_s(\catK) := \on{Ob}(\catK)/{\sim}$ obtains a partial order induced via inclusions of the corresponding semiprime ideals. We call $L_s(\catK)$ the \textit{(semiprime) object poset} of $\catK$.
    \end{remark}

    \begin{remark}\hfill
        \begin{enumerate}
            \item As already alluded to, $L_s(\catK)\cong t_s(\catK)$ by $[x] \mapsto \langle x\rangle_s$.
            \item One can play a similar game by looking at the principal parts of the lattices $T(\catK)$ or $T_r(\catK)$ and define principal parts $t(\catK)$ or $t_r(\catK)$.
            Alternatively, one can think of these in terms of objects up to equivalence relation using $\langle-\rangle$ and $\langle-\rangle_r$ instead of $\langle-\rangle_s$ to define $L(\catK)$ and $L_r(\catK)$ as was done in \cite{Kra24, Mil25}. Since our main goal is the classification of semiprime ideals, we will, for the most part, not consider non-semiprime ideals.
            \item For radical thick $\otimes$-ideals, the situation behaves as it does in the symmetric case. Krause shows that for radical thick $\otimes$-ideals, $L_r(\catK)$ is exactly the compact part of $T_r(\catK)$, hence $T_r(\catK)$ is a coherent frame \cite[Proposition 14]{Kra24}. Miller concludes that $\Spc\cp(\catK) = \Spc(L_r(\catK))^\vee$, hence one has Balmer classification of radical thick $\otimes$-ideals in terms of $\Spc\cp(\catK)$, and $\Spc\cp(\catK)$ is spectral. However, $\Spc\cp(\catK)$ can be empty, see \cite[Theorem 8.5]{Mil25} - this corresponds to $T_r(\catK)$ having one element, $\catK$.
        \end{enumerate}
    \end{remark}

    In order to understand the principal part, we need to understand the joins and meets of principal semiprime $\otimes$-ideals.

    \begin{lemma}\label{lem:distributivityofprincipals}
        Let $x, y \in \catK$. Then in $T_{s}(\catK)$, we have an equality \[\langle x \oplus y \rangle_s  = \langle x\rangle_s \vee \langle y \rangle_s .\]
        Moreover, we have an inclusion
        \[\langle x \otimes y\rangle_s  \subseteq  \langle x\rangle_s \wedge \langle y \rangle_s ,\]

        and equality holds for all elements in $\catK$ if and only if all semiprime ideals of $\catK$ are radical if and only if all primes of $\catK$ are completely prime.
    \end{lemma}
    \begin{proof}
        First, since $x, y \in \langle x \oplus y \rangle_s $ by thickness, $\langle x \rangle_s  \vee \langle y \rangle_s  \subseteq \langle x \oplus y \rangle_s $. On the other hand, $x \oplus y \in \langle x \rangle_s  \vee \langle y \rangle_s $, hence $\langle x \oplus y\rangle_s  \subseteq \langle x \rangle_s  \vee \langle y \rangle_s $, as desired.

        For the latter statement, the inequality is obvious. If equality holds, then in particular, $\langle x \otimes x \rangle_s  = \langle x \rangle_s  \wedge \langle x \rangle_s  = \langle x \rangle_s $ implying all semiprime ideals are radical.
        If all semiprime ideals are radical, then $t_s(\catK) = t_r(\catK)$, therefore $T_r(\catK) = T_s(\catK)$ and $\Spc(\catK) = \Spc\cp(\catK)$ by \cite[Proposition 6.2]{Mil25}; so all primes are completely prime.
        Lastly, suppose all primes are completely prime, and let $z\in \langle x\rangle_s\wedge \langle y\rangle_s$.
        Suppose $z \notin \langle x \otimes y\rangle_s$, then there exists a completely prime $\catP$ with $z\notin \catP \ni x\otimes y$.
        As $\catP$ is completely prime, $x$ or $y$ is contained in $\catP$; but as $z\in \langle x\rangle_s\wedge \langle y\rangle_s$ either implies $z\in \catP$, a contradiction.
    \end{proof}

    \begin{remark}
        For radical thick $\otimes$-ideals, joins and meets mirror the symmetric case, $ \langle x\rangle_r \vee \langle y \rangle_r = \langle x \oplus y \rangle_r$ and  $\langle x\rangle_r \wedge \langle y \rangle_r = \langle x \otimes y\rangle_r = \langle y \otimes x \rangle_r$, as is clear by the proof of \Cref{lem:distributivityofprincipals}.
    \end{remark}

    \subsection{Principal closure}

    Thus $t_s(\catK)$ has joins but not necessarily meets, since the intersection of two principal (or equivalently, finitely generated) semiprimes may not again be principal.
    Indeed,
    \[\langle x\rangle_s  \cap \langle y \rangle_s  =\langle x \otimes \catK \otimes y \rangle_s \] in the poset of semiprime ideals, and in general the right hand side may not be finitely generated, see \Cref{obs:Ts_can_be_ugly}.
    Therefore we consider the following class of monoidal triangulated categories for which this property does hold.

    \begin{definition}
        We say $\catK$ is \textit{principally closed} if for all $x, y \in \catK$, $\langle x\rangle_s  \cap \langle y \rangle_s  = \langle z \rangle_s $ for some $z \in \catK$.
    \end{definition}

    This condition holds in the known noncommutative cases where one has Balmer classification of thick $\otimes$-ideals: when $\Spc(\catK) = \Spc\cp(\catK)$ \cite{MR23}, when $\catK$ has a thick generator \cite{NVY23}, and when $\Spc(\catK)$ is Noetherian \cite{Row24}.

    \begin{example}\label{ex:princ_clos}
        \hfill
        \begin{enumerate}
            \item If $\catK$ is a tensor triangulated category, or more generally the tensor product is (non-functorially) braided, then the primes equal the completely primes and $\catK$ is principally closed by \Cref{lem:distributivityofprincipals}.
            \item More generally, if every semiprime ideal of $\catK$ is radical, then $\catK$ is principally closed also by \Cref{lem:distributivityofprincipals}.
            \item If $\catK$ has a thick generator $g$ then $\catK$ is principally closed. Indeed, in this case we have \[\langle x\rangle_s  \cap \langle y \rangle_s   = \langle x \otimes \catK \otimes y\rangle_s  = \langle x \otimes g\otimes y\rangle_s .\]
            \item If $\Spc(\catK)$ is Noetherian, then $\catK$ is principally closed; this follows from \Cref{cor:ncspcspectralspace} in the sequel and Rowe \cite[Theorem 4.9]{Row24}.
            \item\label{item:princ_clos_pt_Noeth} If $\Spc(\catK)^\nu$ is Noetherian, then $\catK$ is principally closed. Indeed, for any $x,y\in\catK$ the set \[\{\supp(\langle x \otimes S\otimes y\rangle_s) \mid S \subseteq \Ob(\catK) \text{ finite}\}\] of quasi-compact open sets of $\Spc(\catK)^\nu$ has a finite subcover.
        \end{enumerate}
    \end{example}

    If $\catK$ is principally closed, then $t_s(\catK)$ is well-behaved: it is distributive.

    \begin{lemma}\label{lem:distributivitygivenprincipalclosure}
        Let $\catK$ be principally closed. Then $t_s (\catK)$ is a bounded distributive lattice.
    \end{lemma}
    \begin{proof}
        If $\catK$ is principally closed, $t_s(\catK)$ is a bounded sublattice of the (spatial) frame $T_s(\catK)$.
        Consequently it inherits distributivity from the latter.
    \end{proof}

    \subsection{Compact elements in the lattice of semiprimes}

    Our desire, motivated by \Cref{prop:dist_lattice=coherent_frame}, is that $t_s(\catK)$ is the sublattice of compact objects of $T_s(\catK)$.
    We determine exactly when this occurs.

    \begin{definition}
        Let $x \in \catK$. We say \textit{$x$ has compact detection} in $\catK$ if and only if there exists an element $s_x \in \catK$ for which the following holds: given any semiprime thick $\otimes$-ideal $\catI \subseteq \catK$, we have $x \in \catI$ if and only if $x \otimes s_x \otimes x \in \catI$; equivalently, $\langle x\rangle_s = \langle x \otimes s_x \otimes x\rangle_s$.
        We say the element $s_x$ \textit{witnesses compact detection} for $x$.
        We say \textit{$\catK$ has compact detection} if all elements of $\catK$ have compact detection.
    \end{definition}

    Compact detection is a very mild assumption; however, it is not automatic.
    We give some examples and a non-example.

    \begin{example}\label{ex:comp_det}
        \hfill
        \begin{enumerate}
            \item If $\catK$ is rigid (or half-rigid), then $\catK$ has compact detection by choosing $s_x = x^*$ or ${}^*x$ for every element $x \in \catK$. Indeed, one has by (half)-rigidity that $x \otimes x^* \otimes x$ admits $x$ as a direct summand, or similarly for ${}^*x$.

            \item If every semiprime thick $\otimes$-ideal of $\catK$ is radical (e.g., if $\catK$ is symmetric), then $\catK$ has compact detection by choosing $s_x = 1$ for every element $x \in \catK.$ Indeed, by radicalness, one has $x \otimes 1 \otimes x \cong x \otimes x \in \catI$ if and only if $x \in \catI$.

            \item If $\Spc(\catK)$ is Noetherian, then $\catK$ has compact detection by \Cref{cor:ncspcspectralspace} below and \cite[Theorem 4.9]{Row24}.
        \end{enumerate}
    \end{example}

    We note that the following non-example was obtained via generative AI. We take full responsibility for its correctness. To our knowledge, the example is not immediately attributable to any previous results in the literature.

    \begin{example}\label{ex:non-compact_det}
        Let $\bbM = \langle a, b\rangle$ denote the free monoid generated by 2 elements, $k$ any field, and let $\mathsf{grmod}_\bbM(k)$ denote the semisimple monoidal category of $\bbM$-graded vector spaces (cf.\ \Cref{ex:non_semiprimes_ex}). The monoidal structure is given by Day convolution: for $\bbM$-graded vector spaces $v,w \in \mathsf{grmod}_\bbM(k)$ and $\gamma \in \bbM$, \[(v \otimes w)_\gamma := \bigoplus_{\mu\nu = \gamma} v_\mu \otimes w_\nu,\]
        where the sum is finite as each $\gamma$ only has finitely many factorizations.
        The simple objects are the twists $k(\gamma)$, i.e.\ the graded modules satisfying $k(\gamma)_\gamma = k$ and $k(\gamma)_\mu = 0$ for all $\mu \neq \gamma$. In particular, $k(1)$ is the tensor unit of $\mathsf{grmod}_\bbM(k)$ and \[k(\gamma) \otimes k(\delta) \cong k(\gamma\delta).\]

        Let $\catK := D^b(\mathsf{grmod}_\bbM(k))$, we identify objects of $\mathsf{grmod}_\bbM(k)$ as objects of $\catK$ by viewing them as chain complexes concentrated in homological degree 0.
        The category $\catK$ is also semisimple and the simples are the shifts of the $k(\gamma)$s.

        We claim $k(a)$ does not have compact detection in $\catK$.
        To this end, it suffices to show for any $s\in \catK$ that there exists a semiprime thick $\otimes$-ideal $\catJ$ with $k(a)\notin \catJ$ but $k(a)\otimes s \otimes k(a)\in \catJ$.
        Let $s\in\catK$ be arbitrary. Then $s \cong \oplus_{\delta \in M} k(\delta)[i_\delta] $ for some finite subset $M \subset \bbM$ and integers $i_\delta$.
        Pick $n > \max_{\delta \in M}|\delta|$, where $|\delta|$ denotes the word length of $\delta$.
        Set \[P_{n} := \{\delta \in \bbM \mid \delta \text{ contains a factor $a\mu a$ with $|\mu| \leq n$} \}\] and let $\catP_n$ be the smallest triangulated subcategory of $\catK$ containing $\{ k(\delta) \mid \delta \in P_n \}$.
        As $P_n$ is a prime\footnote{Being prime is equivalent to its complement $C:=\bbM\setminus P_n$ being an $m$-system, i.e.\ for any  $\mu,\nu\in C$ there exists a $\gamma\in \bbM$ for which $\mu\gamma\nu\in C$. Indeed, taking $\gamma=b^n$ does the trick as $\mu b^n\nu\in C$.}
        (monoid) ideal of $\bbM$, it follows that $\catP_n$ is in fact a prime (hence semiprime) thick $\otimes$-ideal.
        Now, by construction, we have that $k(a) \otimes s \otimes k(a) \in \catP_n$, but $k(a) \notin \catP_n$. As $s$ was arbitrary, we conclude $k(a)$ does not satisfy compact detection, so neither does $\catK$.
    \end{example}

    It turns out compact detection is the exact condition we need.

    \begin{proposition}\label{prop:characterizingcompactelts}
        Let $\catK$ be a monoidal triangulated category.
        The compact elements of the lattice of semiprime thick $\otimes$-ideals $T_s(\catK)$ are precisely the principal ideals $\langle x\rangle_s $ if and only if $\catK$ has compact detection.
    \end{proposition}

    We prove two lemmas first.

    \begin{lemma}\label{lem:compacts_principal}
        A compact semiprime ideal $\catI \in T_s(\catK)$ must be principal.
    \end{lemma}
    \begin{proof}
        Suppose $\catI \in T_s(\catK)$ is compact.
        As $\catI=\vee_{x\in\catI}\langle x\rangle_s$ compactness gives a finite subset $A \subseteq \catI$ with $\catI = \vee_{a\in A} \langle a\rangle_s=\langle \oplus_{a\in A} a\rangle_s$ and so is principal.
    \end{proof}

    \begin{lemma}\label{lemma:compact}
        Suppose $\catK$ has compact detection.
        For any thick $\otimes$-ideal $\catI$ with $x \in \langle \catI \rangle_s$, there exists $y \in \catK$ witnessing compact detection for $x$ such that $x \otimes y \otimes x\in \catI$.
        Consequently, if $S \subseteq \Ob(\catK)$ is a set of objects of $\catK$ with $x \in \langle S\rangle_s$, there exists a finite subset $S' \subseteq S$ such that $x \in \langle S' \rangle_s$.
    \end{lemma}
    \begin{proof}
        If $x\in \catI$, any $y$ witnessing compact detection for $x$ satisfies $x\otimes y\otimes x\in \catI$. Thus, we may assume $x\notin\catI$.
        Let $x_1 := x$ and pick an $a_1 \in \catK$ witnessing compact detection for $x_1$, which exists by assumption, and put $x_2 := x_1 \otimes a_1 \otimes x_1$.
        Then, inductively for $n > 1$, define $x_{n+1} := x_{n} \otimes a_{n} \otimes x_{n}$ where $a_n$ is chosen to witness compact detection for $x_n$.
        By construction, for any semiprime $\otimes$-ideal $\catJ$, $x \in \catJ $ if and only if $x_n \in \catJ $ for any $n> 1$. Thus as every $x_i$ is of the form $x\otimes s_i\otimes x$, necessarily $s_i$ witnesses compact detection for $x$.

        The set $ \catM := \{x_1, x_2,x_3,\dots\}$ is a nc-multiplicative in the sense of \cite[Appendix A]{NVY23}. Suppose for contradiction that $\catM \cap \catI = \emptyset$. Then by \cite[Lemma A.1.1]{NVY23} there exists a prime $\catP \in \Spc(\catK)$ such that $\catP \supseteq \catI$ and $\catM \cap \catP = \emptyset$.
        However, since $x\in\langle \catI\rangle_s$ by assumption, $x \in \catP$ for all primes $\catP \supseteq \catI$ by semiprimality of $\catP$.
        As $x \in \catM$, this is a contradiction.
        Thus, there exists some $x_i \in \catI$ as desired.

        For the final statement, if $x\in\langle S\rangle_s$ pick a $y$ as above with $x\otimes y\otimes x\in\langle S\rangle$.
        Since every element of $\langle S\rangle$ can be built from elements of $S$ in finitely many steps via tensoring on the left and right, taking summands, shifts, and cones, we may take $S'$ to be the finite set of elements of $S$ involved in the construction of $x \otimes y \otimes x$.
        Thus $x \otimes y \otimes x\in \langle S'\rangle$ and consequently $x\in \langle x\rangle_s = \langle x \otimes y \otimes x\rangle_s \subseteq \langle S'\rangle_s$.
    \end{proof}

    \begin{proof}[Proof of \Cref{prop:characterizingcompactelts}]
        By \Cref{lem:compacts_principal} it suffices to show that every principal is compact if and only if $\catK$ has compact detection.

        Suppose every principal is compact. Fix $x\in\catK$.
        We have $x\in \langle x\otimes \catK\otimes x\rangle_s=\vee_{y\in\catK}\langle x\otimes y\otimes x \rangle$, so by assumption there exists a finite subset $S \subseteq \catK$ for which $\langle x\rangle_s = \vee_{y\in S}\langle x\otimes y\otimes x \rangle$.
        Taking the sum of the elements in $S$ gives an element witnessing compact detection for $x$. Hence, $\catK$ has compact detection.

        Conversely, suppose $\catK$ has compact detection. Pick $x \in \catK$ and suppose $\{\catI_i\}_{i\in I}\subseteq T_s(\catK)$ is a collection of semiprime thick $\otimes$-ideals with $\langle x\rangle_s \leq \vee_{i\in I} \catI_i=\langle \cup_i \catI_i\rangle_s$.
        By \Cref{lemma:compact} there exist a finite subset $S\subset  \cup_i \catI_i$ with $x\in \langle S\rangle_s$; clearly, this implies there is a finite subset $I'\subset I$ with $\langle x\rangle_s \leq \vee_{i\in I'} \catI_i$ showing $\langle x\rangle_s$ is compact.
    \end{proof}

    The following corollary can be observed directly from the proof of \Cref{prop:characterizingcompactelts}. It is unclear if the converse holds in full generality.
    \begin{corollary}\label{cor:noncompacts}
        If $\langle x\rangle_s$ is a compact element of $T_s(\catK)$, then $x \in \catK$ has compact detection.
    \end{corollary}

    \begin{corollary}\label{cor:semiprimeimpliescompactdetection}
        Suppose all ideals of $\catK$ are semiprime (e.g., $\catK$ is rigid). Then $\catK$ has compact detection.
    \end{corollary}
    \begin{proof}
        By assumption $T_s(\catK)=T(\catK)$ the lattice of thick $\otimes$-ideals.
        Hence, by \Cref{prop:characterizingcompactelts}, it suffices to show that the compact elements of $T(\catK)$ are precisely the principal ideals $\langle x\rangle$.
        This is well-known.
        The same argument as \Cref{lem:compacts_principal} shows that any compact needs to be principal and the fact that principals are compact was already used in \Cref{lemma:compact}; i.e.,\ any object in $\langle S\rangle$ can be built from elements of $S$ in finitely many steps via tensoring on the left and right, taking summands, and taking cones.
    \end{proof}

    \begin{remark}\hfill
        \begin{enumerate}
            \item The converse of \Cref{cor:semiprimeimpliescompactdetection} does not hold; $\catK$ having compact detection does not imply all thick $\otimes$-ideals of $\catK$ are semiprime. This can easily be seen in the tensor triangular setting: if $\catK$ is a tensor triangulated category, then $T_s(\catK) = T_r(\catK)$ and $t_s(\catK) = t_r(\catK)$.
            Hence $\catK$ always has compact detection, but there are small tensor triangulated categories with non-radical thick $\otimes$-ideals.
            \item Gratz--Stevenson's result \cite[Theorem 6.4.5]{GS23} can be viewed, after a bit of yoga, as a reformulation of \Cref{cor:classification_for_spatial_frames} with the assumption that all ideals are semiprime. It is certainly the predecessor, chronologically.
            The dual topology provided in \cite[Theorem 6.4.5]{GS23} is defined by declaring the closed sets to be generated by the quasi-compact open subsets of $\pt(T_s(\catK))$.
            Since their assumptions imply $\catK$ has compact detection, the quasi-compact open subsets are precisely the supports of objects $\supp(x)$.
            This coincides with the topology of $\Spc(\catK)^\nu$.
            \item If follows immediately that when $\catK$ satisfies compact detection that $T_s(\catK)$ is algebraic, as the principal ideals always generate under join.
            Of course, it may be possible that $T_s(\catK)$ is algebraic without compact detection holding, but we know of no examples.

            \item If $\catK$ satisfies compact detection and principal closure, then one can say more about the element generating the intersection; namely, for all $x,y \in \catK$, if $\langle x\rangle_s \cap \langle y \rangle_s$ is principal and hence compact $\langle x\rangle_s \cap \langle y \rangle_s = \langle x \otimes \catK \otimes y\rangle_s$ implies there exists a $z\in\catK$ with $\langle x \rangle_s \cap \langle y \rangle_s = \langle x \otimes z \otimes y\rangle_s$ (cf.\ the proof of \Cref{prop:characterizingcompactelts}).
        \end{enumerate}

    \end{remark}

    \subsection{Putting things together}\label{subsec:control_pp}

    We have everything we need to characterize when the lattice $T_s(\catK)$ is coherent in such a way that the usual classification in terms of Thomason subsets of the noncommutative Balmer spectrum $\Spc(\catK)$ goes through.

    \begin{theorem}\label{thm:distlattice}
        Let $\catK$ be a monoidal triangulated category. The following are equivalent.
        \begin{enumerate}
            \item $\catK$ is principally closed and has compact detection;
            \item $T_s(\catK)$ is a coherent frame with sublattice of compact elements $t_s(\catK)$.
        \end{enumerate}
    \end{theorem}
    \begin{proof}
        The principals always generate under join, i.e.,\ for any $\catI\in T_s(\catK)$ one has $\catI=\vee_{x\in\catI}\langle x\rangle_s$.
        Moreover, $t_s(\catK)$ are the compacts if and only if compact detection holds by \Cref{prop:characterizingcompactelts} and these form a sublattice if and only if principal closure holds by \Cref{lem:distributivitygivenprincipalclosure}.
    \end{proof}
    \begin{remark}
        One can also rephrase this in terms of the object poset.
        The conditions of the theorem are equivalent to $T_s(\catK)\isoto \Id(L_s(\catK))$ under the map $\catI \mapsto \{[x] \in L_s(\catK) \mid \langle x\rangle_s \subseteq \catI\}$.
    \end{remark}

    \begin{remark}

        If $T_s(\catK)$ is coherent and satisfies compact detection, then it automatically satisfies principal closure, since coherent frames always satisfy that their compact objects form a bounded, distributive lattice.
        In principle, it seems conceivable that $T_s(\catK)$ could be a coherent frame whose compact elements are properly contained in $t_s(\catK)$ - this cannot happen if $\catK$ has a commutative tensor product. It would be interesting to know if there are any interesting examples of essentially small monoidal triangulated categories $\catK$ where this occurs. Our expectation is that examples may be rather pathological.
    \end{remark}

    \begin{corollary}\label{cor:ncspcspectralspace}
        Let $\catK$ be a monoidal triangulated category. The following are equivalent:
        \begin{enumerate}
            \item\label{item:ncspcspectralspace1} $\catK$ is principally closed and has compact detection;
            \item\label{item:ncspcspectralspace2} The noncommutative Balmer spectrum $\Spc(\catK)$ is a spectral space with quasi-compact opens given by complements of supports,  \[K^\circ(\Spc(\catK)) = \{\supp(x)^c \mid x \in \catK\}.\]
        \end{enumerate}
    \end{corollary}
    \begin{proof}
        For \ref{item:ncspcspectralspace1} implies \ref{item:ncspcspectralspace2}, we have $T_s(\catK)=\Id(t_s(\catK))$ by \Cref{thm:distlattice} and therefore $\Spc(\catK) = \Spc(t_s(\catK))^\vee$, where the latter is the spectrum of the bounded distributive lattice as defined in \Cref{sec:Spc_BDLat}, whose quasi-compacts are characterized by \Cref{prop:quasicompacts}.

        Conversely, suppose \ref{item:ncspcspectralspace2}.
        Then, the Hochster dual of $\Spc(\catK)$ is exactly $\pt(T_s(\catK))$.
        As this is the (genuine) Hochster dual of a spectral space, this shows $\Omega(\pt(T_s(\catK)))$ is a coherent frame with compact elements precisely the supports $\{\supp(x) \mid x \in \catK\}$.
        As $T_s(\catK)=\Omega(\pt(T_s(\catK)))$ by \Cref{cor:classification_for_spatial_frames}, with principals corresponding to the supports of elements, the principals are exactly the sublattice of compact objects, as desired.
    \end{proof}

    Furthermore, Balmer's classification theorem holds in its original form.

    \begin{corollary}[{cf.\ \cite[Theorem 4.10]{Bal05}}]\label{cor:balmerclassification}
        If $\catK$ has principal closure and compact detection then we have an order-preserving bijection $T_s(\catK) \cong \on{Th}(\Spc(\catK))$ induced by the following assignments:
        \[\catI \mapsto \bigcup_{x \in \catI} \supp(x) \in \on{Th}(\Spc(\catK))\] \[Z  \mapsto \{x \in \catK \mid \supp(x) \subseteq Z\} \in T_s(\catK).\]
    \end{corollary}
    \begin{proof}
        Follows directly from \Cref{thm:balmerclassification_without_conditions} and \Cref{cor:ncspcspectralspace}.
    \end{proof}

    \begin{remark}\label{rmk:examplesofclassification}
        \Cref{cor:balmerclassification} subsumes the following cases (listed in \cite[Theorem 2.1]{HV25}), as all of the these have principal closure and compact detection (cf.\  \Cref{ex:princ_clos,ex:comp_det}):
        \begin{enumerate}
            \item When $\otimes$ is commutative \cite[Theorem 4.10]{Bal05};
            \item When all prime ideals of $\catK$ are completely prime \cite[Theorem 3.11]{MR23};
            \item When $\Spc(\catK)$ is Noetherian \cite[Theorem 4.9]{Row24};
            \item When $\catK$ has a thick generator \cite[Theorem A.7.1]{NVY23}.
        \end{enumerate}
        As we verify in \Cref{sec:centralgeneration}, \Cref{cor:balmerclassification} also holds for monoidal triangulated categories with the (weak) central generation property and compact detection (the latter being automatic when half-rigid). Moreover, \cite[Example 7.7]{HV25} is an example of a rigid, hence compact-detecting, monoidal triangulated category that is not principally closed, since its noncommutative Balmer spectrum is non-spectral (see \Cref{obs:Ts_can_be_ugly}).
    \end{remark}

    We also obtain a strengthening of \cite[Theorem 6.2.1]{NVY22}, on detecting Balmer spectra via classifying support data. In particular, this statement extends beyond the Noetherian Balmer spectrum setting, although it still lives within the coherent frame setting.
    As the results are stated in terms of \emph{closed} support data we quickly translate the required language from \Cref{sec:support} into that language.

    For ease, we assume $T_s(\catK)$ is coherent with compact part the principals $t_s(\catK)$.
    The definition of a weak support datum as in \cite[Definition 4.4.1]{NVY22} then corresponds to the following.\footnote{\label{footnote:tweak_weak}Technically this only holds under the conditions imposed in \cite[Theorem 6.2.1]{NVY22}, namely that all ideals are semiprime and the principals have closed values. However, what we present is correct from the lattice point of view, which is equivalent to \cite[Definition 4.4.1]{NVY22} with the modification that objects have values in closed subsets and (e) is required to hold only for semiprime ideals.}
    \begin{definition}\label{def:weak_supp}
        Assume $\catK$ has principal closure and compact detection. A weak support datum $(X,\sigma)$ for $\catK$ is a bounded lattice morphisms $\sigma\colon t_s(\catK) \to \on{Cl}(X)$ to the set of closed sets of $X$.
        Analogously to \Cref{def:supportdatum} we say $(X,\sigma)$ satisfies:
        \begin{enumerate}
            \item \textit{injectivity} if $\sigma$ is an injective map of sets;
            \item \textit{faithfulness} if $\sigma(x) = \emptyset$ implies $x = 0$;
            \item \textit{realization} if the Thomason closed\footnote{$A\subseteq X$ is \emph{Thomason closed}, if it is open for $X^\vee$ and closed for $X$; equivalently $A$ is the complement of a quasi-compact open.} subsets lie in the image of $\sigma$;
            \item \textit{Noetherian realization} if $\sigma$ is surjective.
        \end{enumerate}
    \end{definition}
    Any weak support datum $(X,\sigma)$ for $\catK$ yields an open support datum $(X,\tau)$ for $t_s(\catK)^{\on{op}}$ by putting $\tau(Z):=X\setminus Z$, and vice versa.
    Moreover, the properties of \Cref{def:weak_supp} correspond to the respective ones for open support data (\Cref{def:supportdatum}).
    Hence the results from \Cref{sec:support} translate immediately.

    \begin{corollary}\label{thm:classifyingsupportdatumforclosedsupport}
        Let $\catK$ be a monoidal triangulated category which is principally closed and has compact detection, and let $(X, \sigma)$ be a weak support datum for $\catK$ (see \Cref{def:weak_supp}).
        \begin{enumerate}
            \item\label{item:classifyingsupportdatum1} If $(X,\sigma)$ is injective and realizing, then $(X,\sigma)$ is a classifying support datum, hence $(X,\sigma) \cong (\Spc(\catK),\supp)$.
            \item\label{item:classifyingsupportdatum2} Assume $\catK = \catT^c$ for a rigidly-compactly-generated monoidal triangulated category $\catT$ and that $(X,\sigma)$  and that $(X,\sigma)$ is an extended weak support datum for $\catT$. If $(X,\sigma)$ is \textit{faithful} and realizing, then $(X,\sigma) \cong (\Spc(\catK), \supp)$.
        \end{enumerate}
        In either case, if $(X,\sigma)$ is Noetherian-realizing, then $\Spc(\catK)$ is Noetherian.
    \end{corollary}
    \begin{proof}
        We have the identification $\Spc(\catK) = \Spc(t_s(\catK))^\vee = \Spc(t_s(\catK)^{\on{op}})$ arising from the identification $T_s(\catK) = \Id(t_s(\catK))$ of \Cref{prop:dist_lattice=coherent_frame} and \Cref{thm:distlattice}.
        As any weak support datum for $\catK$ is the equivalent to an open support datum for $t_s(\catK)^{\on{op}}$, as noted above, \ref{item:classifyingsupportdatum1} is a direct translation of \Cref{prop:classifying}.
        For \ref{item:classifyingsupportdatum2} one may apply \cite[Theorem 5.3.1]{NVY22}, which asserts that faithfulness and injectivity are equivalent.
        The final statement arises from \Cref{cor:noetherianrealizing}.
    \end{proof}

    Again, we stress that there is no relation between $\Spc(\catK)$ Noetherian and $\Spc(\catK)^\vee$ Noetherian without additional assumptions. This fact also should be noted for Barthel's adaptation of Cohen's theorem, which also applies to the noncommutative setting.

    \begin{corollary}\label{cor:Barthel_translated}
        Let $\catK$ be a monoidal triangulated category.
        Then $\Spc(\catK)$ is weakly Noetherian and every semiprime thick $\otimes$-ideal is principal if and only if $\Spc(\catK)$ is finite. Moreover, $\Spc(\catK)^\nu$ is Noetherian if and only if every semiprime thick $\otimes$-ideal is principal. If any of the above conditions hold, $\catK$ satisfies principal closure and compact detection.
    \end{corollary}
    \begin{proof}
        The first statement follows immediately from \Cref{thm:barthelthm}, once we verify that either condition implies principal closure and compact detection, as then $\Spc(\catK)=\Spc(\catL)^\vee$ for the bounded distributive lattice $\catL=t_s(\catK)$.
        \begin{itemize}
            \item  If every semiprime thick $\otimes$-ideal is principal, then it immediately follows that $t_s(\catK) = T_s(\catK)$.
            Hence $t_s(\catK)$ is clearly a lattice, i.e.,\ principal closure holds.
            Moreover, as every semiprime $\otimes$-ideal is finitely generated every semiprime $\otimes$-ideal is compact, so compact detection holds by \Cref{prop:characterizingcompactelts}.
            \item If $\Spc(\catK)$ is finite, then $\Spc(\catK)^\vee$ is also finite and so Noetherian.
            Hence, $\catK$ satisfies principal closure by \Cref{ex:princ_clos}\ref{item:princ_clos_pt_Noeth} and, as $T_s(\catK)$ is necessarily finite, every element is compact so compact detection holds.
        \end{itemize}

        For the second statement, if $\Spc(\catK)^\nu$ is Noetherian, every element of $\Omega(\Spc(\catK)^\nu)\allowbreak=T_s(\catK)$ is compact. Hence, by \Cref{lem:compacts_principal} any semiprime thick $\otimes$-ideal is principal.
        Conversely, if every semiprime thick $\otimes$-ideal is principal, the first part showed that $\catK$ satisfies compact detection.
        Thus $t_s(\catK)=T_s(\catK)=\Omega(\Spc(\catK)^\nu)$ consists entirely of compact elements, so necessarily $\Spc(\catK)^\nu$ is Noetherian (use e.g.,\ \cite[Proposition 7.13]{BF11}).
    \end{proof}

    In particular, $\Spc(\catK)^\nu$ is rather complicated: as soon as $\Spc(\catK)$ has non-principal primes (which usually happens in practice), $\Spc(\catK)^\nu$ is non-Noetherian.

    \subsection{A note on functoriality}

    Stone duality is an antiequivalence, hence for every homomorphism of coherent frames $f\colon  \catF \to \catF'$ one obtains a corresponding continuous homomorphism on the space of points $f^*: \pt(\catF') \to \pt(\catF)$, and vice versa. Similarly, given any morphism of bounded distributive lattices $l: \catL \to \catL'$, one obtains a corresponding morphism on their ideal lattices $\Id{(l)}: \Id(\catL) \to \Id(\catL')$, and conversely if one has a compact-preserving morphism of coherent frames $f\colon  \catF \to \catF'$, it restricts to a morphism on their compact elements, $f^c: \catF^c \to (\catF')^c$, necessarily a morphism of bounded distributive lattices.

    The reason why functoriality generally fails for $\Spc(-)^\nu$ is clear: an exact monoidal functor $F\colon  \catK \to \catK'$ may not yield a meet preserving map on lattices.
    Any such functor induces a map of \textit{posets} $\hat{F}: T_s(\catK) \to T_s(\catK'),\ \catI \mapsto \langle F(\catI)\rangle_s$ which respects joins.
    However without strong hypotheses, such as $T_s(\catK') = T_r(\catK')$ (i.e., $\catK'$ satisfies the tensor product property) or $F$ essentially surjective, $\hat{F}$ may in general fail to preserve meets.
    The same remains true for the restricted morphisms $\hat{F}: t_s(\catK) \to t_s(\catK')$, even when $\catK$ and $\catK'$ are principally closed.
    Hence for functoriality, it seems the lattice perspective offers little more than what was already deduced in \cite[Section 5]{Mil25}.

    \section{Applications}\label{sec:ex}
    \subsection{(Weak) central generation implies principal closure}\label{sec:centralgeneration}

    We consider the (weak) \textit{central generation hypothesis} considered in \cite{NP23} for stable module categories and \cite{NVY25} in greater generality, and show that if $\catK$ is such a monoidal triangulated category with compact detection (e.g., $\catK$ is rigid or half-rigid), then $\catK$ is principally closed.

    \begin{definition}
        Let $\catK$ be a monoidal triangulated category.
        \begin{enumerate}
            \item We say $\catK$ satisfies the \textit{central generation hypothesis} if all thick $\otimes$-ideals of $\catK$ are generated by collections of central elements, i.e., elements $x \in \catK$ such that $x \otimes y \cong y \otimes x$ (possibly non-functorially) for all objects $y \in \catK$. If $\catT$ is a `big', i.e.,\ rigidly-compactly-generated, monoidal triangulated category, we say $\catT$ satisfies the central generation hypothesis if its compact part $\catT^c$ does.
            \item We say $\catK$ satisfies the \textit{weak central generation hypothesis} if there exists a collection of objects $I \subseteq \on{Ob}(\catK)$ thickly generating $\catK$ for which all thick $\otimes$-ideals of $\catK$ are generated by collections of objects which \emph{centralize} $I$, that is, generated by objects $x$ satisfying $x \otimes y \cong y \otimes x$ for all $y \in I$.
            We refer to $I$ as a \textit{centralizing set}.
            Of course, taking $I = \Ob(\catK)$ for $\catK$ satisfying the central generation hypothesis shows it also satisfies weak central generation hypothesis. Again, if $\catT$ is a big monoidal triangulated category, we say $\catT$ satisfies the weak central generation hypothesis if $\catT^c$ does.
        \end{enumerate}
    \end{definition}

    \begin{example}\label{ex:centgencats}
        Negron--Pevtsova consider a specific example of weak central generation restricted to stable module categories in \cite{NP23}. In this case, they consider objects which centralize the simple objects of the corresponding module category. The authors verify weak central generation \cite[Section 8]{NP23} for the following stable module categories:

        \begin{enumerate}
            \item Bosonized quantum complete intersections;
            \item Small quantum Borels $u_q(B)$ in type A, at arbitrary odd order $q$;
            \item Drinfeld double $\calD(B_{(1)})$ for $B$ a Borel in an almost simple algebra group $\bbG$ over $\overline{\bbF}_p$ at arbitrary $p$.
        \end{enumerate}

        (Strong) central generation plays a key role in Nakano--Vashaw--Yakimov's extension of the homological spectrum, first constructed by Balmer in the symmetric case \cite{Bal20}, to the noncommutative setting \cite{NVY25}. In this case, if $\catK$ satisfies central generation, then the so-called ``homological primes'' of $\catK$ coincide with Balmer's original notion. Additionally, one may translate \cite[Lemmata 6.6, 6.7]{NP23} over to the more general monoidal triangulated category setting and verify they hold - in particular, $\Spc(\catK)$ is multiplicative in the sense of \cite{NP23}.
    \end{example}

    \begin{proposition}\label{thm:weak_central_gen_principal}
        Suppose $\catK$ satisfies weak central generation with centralizing set $I \subseteq \Ob(\catK)$ and has compact detection (e.g.,\ is rigid). Then every principal thick $\otimes$-ideal of $\catK$ is generated by an element centralizing $I$. Consequently, $\catK$ is principally closed.
    \end{proposition}
    \begin{proof}
        Let us first show that for any $x\in\catK$ there exists an $x' \in \catK$ central with respect to $I$ such that $\langle x \rangle_s = \langle x'\rangle_s$. We have by hypothesis that $\langle x \rangle_s = \langle S\rangle_s$ for some (possibly infinite) set $S \subseteq \Ob(\catK)$ central with respect to $I$. \Cref{lemma:compact} implies that there exists a finite subset $S' \subseteq S$ of central elements such that $\langle x \rangle_s = \langle S' \rangle_s$, and taking $x'$ to be the sum of elements in $S'$ shows $\langle x\rangle_s = \langle x' \rangle_s$ with $x'$ central with respect to $I$.

        Now, we have for any $x,y\in\catK$ by choosing $x',y'\in\catK$ central with respect to $I$ as above \[\langle x \rangle_s \cap \langle y \rangle_s = \langle x'\rangle_s \cap \langle y' \rangle_s = \langle \langle x' \rangle_s \otimes \langle y' \rangle_s\rangle_s = \langle x' \otimes \catK \otimes y' \rangle_s = \langle x' \otimes I \otimes y' \rangle_s= \langle x' \otimes y'\rangle_s.\]
        Therefore $\catK$ is principally closed.
    \end{proof}

    Given that central generation contributes to friendly behavior of monoidal triangulated categories (see e.g., \cite[Theorem C]{NVY25}), it is a pertinent question to further understand this property. We note Negron--Pevtsova prove a crucial lemma for detecting weak central generation for finite tensor categories \cite[Lemma 8.10]{NP23}, which they utilize to show the examples in \Cref{ex:centgencats} are centrally generated.

    \subsection{Revisiting crossed product categories}\label{sec:crossedproduct}

    We relate things back to crossed product categories and show when, under some additional hypotheses, these inherit compact detection and principal closure.
    So let $G$ be a group acting on a monoidal triangulated category $\catK$ via monoidal triangulated autoequivalences; see \Cref{sec:groupactions1} for recollections and the definition of the crossed product category $\catK \rtimes G$.

    \begin{lemma}\label{lem:resp_semiprime}
        Every $G$-prime of $\catK$ is semiprime, i.e., an intersection of primes, if and only if the bijection between $G$-ideals of $\catK$ and $\otimes$-ideals of $\catK \rtimes G$ (see \cite[Proposition 5.4]{HV25}) restricts to a bijection between semiprime $G$-ideals and semiprime $\otimes$-ideals of $\catK \rtimes G$.
    \end{lemma}
    \begin{proof}
        As we have a lattice isomorphism between the $G$-ideals of $\catK$ and the  $\otimes$-ideals of $\catK \rtimes G$, it necessarily restricts to an isomorphism between those elements that are meets of meet-prime elements, i.e., the semiprime elements of the lattice. For $\catK \rtimes G$, these are precisely the semiprime thick $\otimes$-ideals, so it suffices to show that every $G$-prime of $\catK$ is semiprime if and only if every semiprime element in the lattice of $G$-ideals of $\catK$ is a semiprime thick $\otimes$-ideal. If every $G$-prime of $\catK$ is semiprime, then since every semiprime element in the lattice is an intersection of $G$-primes, every such element is also an intersection of honest primes, and so a semiprime thick $\otimes$-ideals. Conversely, if there exists a $G$-prime of $\catK$ not a semiprime thick $\otimes$-ideal, then, as it is trivially a semiprime element of the lattice, there exists such an element not a semiprime thick  $\otimes$-ideal.
    \end{proof}
    \begin{remark}
        Of course, the conditions in \Cref{lem:resp_semiprime} are trivially satisfied when all ideals of $\catK$ are semiprime, e.g., if $\catK$ is half-rigid.
        Huang--Vashaw \cite[Proposition 7.2]{HV25} prove a stronger condition under stronger hypotheses; the authors show every $G$-prime is an intersection of a $G$-orbit of a prime, assuming $\Spc(\catK)$ is Noetherian and $\catK$ is half-rigid.
    \end{remark}

    \begin{proposition}
        Let $\catK$ be a monoidal triangulated category and let $G$ be a group acting on $\catK$ via autoequivalences. If all $G$-primes of $\catK$ are semiprime and $\catK$ has compact detection, then $\catK \rtimes G$ has compact detection.
    \end{proposition}
    \begin{proof}
        Let $x \in \catK$ and $g \in G$, and suppose $s_x$ witnesses compact detection for $x$. We claim $s_{x \boxtimes g} := g\inv(s_x) \boxtimes g\inv$ witnesses compact detection for $x \boxtimes g$. Indeed, first note that \[(x \boxtimes g ) \otimes (g\inv(s_x) \boxtimes g\inv) \otimes (x\boxtimes g) \cong (x \otimes s_x \otimes x)\boxtimes g.\]
        Moreover, by \Cref{lem:resp_semiprime}, any semiprime thick $\otimes$-ideal $\catJ = \catI \rtimes G$ of $\catK \rtimes G$ satisfies that $\catI \subseteq \catK$ is semiprime as well. Therefore,  $x \boxtimes g \in \catJ$ if and only if $x \in \catI$ if and only if $x \otimes s_x \otimes x \in \catI$ if and only if $(x \otimes s_x \otimes x)\boxtimes g \in \catJ$, as desired.
    \end{proof}

    Huang--Vashaw's example \cite[Example 7.7]{HV25} demonstrates that a crossed product category formed from a monoidal triangulated category $\catK$ satisfying principal closure need not satisfy principal closure again. However, if $G$ is finite, this can be salvaged.

    \begin{proposition}\label{prop:principalclosureforcrosseds}
        Let $\catK$ be a monoidal triangulated category and let $G$ be a finite group acting on $\catK$ via autoequivalences. If all $G$-primes of $\catK$ are semiprime and $\catK$ satisfies principal closure, then so does $\catK \rtimes G$.
    \end{proposition}
    \begin{proof}
        In fact, the bijection between semiprime $G$-ideals of $\catK$ and semiprime ideals of $\catK\rtimes G$ from \Cref{lem:resp_semiprime} preserves principal ones.
        Indeed, if $\langle x\rangle_s$ is a $G$-ideal of $\catK$, then $\langle x\rangle_s \rtimes G$ is a principal semiprime ideal of $\catK \rtimes G$, with generator $x \boxtimes 1$. Conversely, if $\catI \rtimes G$ is semiprime and principally generated by $x_1 \boxtimes g_1 \oplus \cdots \oplus x_n \boxtimes g_n$, then $\catI$ is principally generated by \[x := \bigoplus_{i=1}^n \bigoplus_{g \in G} g(x_i).\] The result now follows since $(\catI \rtimes G) \cap (\catJ \rtimes G) = (\catI\cap \catJ) \rtimes G$.
    \end{proof}

    \begin{remark}
        The key obstruction preventing this statement from holding in the infinite group case is that $\catK \rtimes G$ may have principal semiprime ideals $\catI \rtimes G$ for which $\catI$ may not be a principal semiprime $G$-ideal of $\catK$. Indeed, for any $x \in \catK$, the semiprime ideal $\catI = \langle g(x) \mid g \in G\rangle_s$ may not be a principal semiprime, but $\catI \rtimes G$ is principal and semiprime, generated by $x \boxtimes 1$. In this case, one cannot use principal closure of $\catK$ to deduce anything about intersections with this principal semiprime.
    \end{remark}

    \subsection{A question of Negron--Pevtsova}

    We conclude by answering a question regarding one-sided thick $\otimes$-ideals posed by Negron--Pevtsova \cite{NP23}.
    Let $\catG$ denote a (generally non-connected) finite group scheme over a perfect field $k$.
    Recall from loc.\ cit.\ that cohomological support for $\on{Coh}(\catG) := \on{mod}(k[\catG])$, where $k[\catG]$ denotes the coordinate algebra of $\catG$, produces one-sided (right) ideals in the stable category.
    For any specialization-closed subset
    \[
    \Theta \subseteq \mathsf{Y} := \Proj(\Ext^\sbull_{\on{Coh}(\catG)}(\bbone, \bbone))
    \]
    one has a corresponding ideal
    \begin{equation}\label{eq:catM}
        \catM_\Theta := \{V \in \stab(\on{Coh}(\catG)) \mid \supp^{coh}_\mathsf{Y}(V) \subseteq \Theta\}.
    \end{equation}
    This assignment produces an injective map $\catM$ from the set of specialization-closed subsets of $\mathsf{Y}$ to the set of thick one-sided $\otimes$-ideals in $\stab(\on{Coh}(\catG))$.

    \begin{question}[{\cite[Question 11.1]{NP23}}]
        Is $\catM$ a bijection? That is, does cohomological support actually classify one-sided $\otimes$-thick ideals?
    \end{question}

    The answer is yes, as we now show.
    We first give a general result on classifying one-sided ideals in crossed product categories $\catK \rtimes G$ for ``nice'' $\catK$.

    \begin{proposition}\label{prop:classificationofonesidedscrossed}
        Let $\catK$ be a monoidal triangulated category  and let $G$ be a group acting on $\catK$ via autoequivalences.
        Furthermore, assume $\catK$ satisfies
        \begin{enumerate}
            \item all right thick $\otimes$-ideals are two-sided, and;
            \item all thick $\otimes$-ideals are semiprime.
        \end{enumerate}
        Then, there is a bijection between the open subsets of $\Spc(\catK)^\nu$ and the thick right $\otimes$-ideals of $\catK \rtimes G$ induced by $U \mapsto \catI_U := \operatorname{add}\{x \boxtimes g \in \catK \rtimes G \mid \supp_\catK(x) \subseteq U\}$.
    \end{proposition}
    \begin{proof}
        By \Cref{cor:classification_for_spatial_frames} we have an isomorphism of lattices $\Omega(\Spc(\catK)^\nu) \cong T(\catK)$.
        Furthermore, \cite[Proposition 8.4(a)]{Mil25} shows that $\catI\mapsto \catI\rtimes G$ induces a bijection between $T(\catK)$ and the lattice of thick right $\otimes$-ideals of $\catK \rtimes G$.
        A straightforward verification shows that the combination of these isomorphisms is given as stated, i.e.\ $U\mapsto \{x \in \catK \mid \supp(x) \subseteq U\}\mapsto \catI_U \subseteq \catK \rtimes G$.
    \end{proof}
    \begin{remark}
        A similar result may be derived for left thick $\otimes$-ideals using \cite[Proposition 8.4(b)]{Mil25}.
    \end{remark}

    \begin{corollary}\label{cor:NP}
        Let $\catG$ be a finite group scheme over a perfect field.
        The map $\catM$ (\Cref{eq:catM}) induces a bijection between specialization-closed subsets of $\mathsf{Y} := \Proj(\Ext^\sbull_{\on{Coh}(\catG)}(\bbone, \bbone))$ and thick right $\otimes$-ideals of $\stab(\on{Coh}(\catG))$.
    \end{corollary}
    \begin{proof}
        There is an equivalence of categories $\stab(\on{Coh}(\catG)) \cong \stab(\on{Coh}(\catG_0)) \rtimes \pi$, where $\catG_0$ is the identity component and $\pi = \catG_{red} \subseteq \catG$ is the (finite) \'{e}tale subgroup; see e.g.\ \cite[Section 9.3]{NVY24}.
        Let $\catK:=\stab(\on{Coh}(\catG_0))$, then $\catK\rtimes \pi \cong \stab(\on{Coh}(\catG))$.
        Moreover, by \cite[Theorem 10.3]{NP23},
        \[
        \Spc(\catK) \cong \mathsf{Y}\quad\text{and}\quad \Spc(\stab(\on{Coh}(\catG))) \cong \mathsf{Y}/\pi.
        \]
        Furthermore, since $\on{Coh}(\catG_0)$ is rigid and generated by the tensor unit, as $k[\catG_0]$ is a unipotent Hopf algebra, all one-sided ideals of $\catK$ are two-sided and semiprime. See \cite[Section 9.3]{NVY24} for details.
        Therefore, by \Cref{prop:classificationofonesidedscrossed}, the thick right $\otimes$-ideals of $\stab(\on{Coh}(\catG))$ are classified by the open subsets of $\Spc(\catK)^\nu$ which we claim are exactly the specialization-closed subsets of $\mathsf{Y}$.
        Indeed, as $\Ext_{\on{Coh}(\catG)}^\sbull(\bbone,\bbone)$ is Noetherian \cite{FS97}, $\mathsf{Y}$ is a Noetherian spectral space.
        Thus, the open subsets of $\Spc(\catK)^\nu = \Spc(\catK)^\vee$ are the Thomason subsets of $\Spc(\catK)$, but these are exactly the specialization-closed subsets by Noetherianity. It is clear, as the supports coincide, that this bijection is precisely the one induced by $\catM$.
    \end{proof}

    \begin{remark}\label{rem:one-sided-class}
        The classification of \Cref{cor:NP} is a highly exceptional case; the question of classifying one-sided thick $\otimes$-ideals from two-sided data (such as $\Spc(\catK)$) is unrealistic in general. For instance, the second author gave an example of a monoidal triangulated category $\catK$ with $\Spc(\catK) = \{0\}$, but which has countably many left and right thick $\otimes$-ideals, all of which are meet-prime in their respective lattices \cite{Mil25b}.

        In general, the question of suitably classifying the one-sided ideals is difficult.
        If one asks for a classification by \textit{quasi-support data} $(X,\sigma)$, a reasonable notion of support data for right thick $\otimes$-ideals where one requires the weaker one-sided multiplicative property $\sigma(x\otimes y) \subseteq \sigma(x)$, the classification becomes `trivial' and so also useless.
        Indeed, \cite[Theorem 7.2]{Mil25} shows that the universal quasi-support datum is the space $\on{Sp}^r(\catK)$ consisting of \textit{all proper} right thick $\otimes$-ideals suitably topologized.
        This is analogous to what happens for the universal support of \cite{BO24} and can also be explained using results of \cite[Section 5]{GS23}. All in all, to get a more useful classification another notion of support is required.

    \end{remark}

    \bibliography{bib}
    \bibliographystyle{alpha}

\end{document}